\documentclass[12pt,a4paper,reqno]{amsart}
\usepackage{amsfonts}
\usepackage{epsfig}
\usepackage{graphicx}
\usepackage{amsmath}
\usepackage{amssymb}
\usepackage{txfonts}
\usepackage{mathrsfs}
\usepackage{color}
\usepackage{bbm}
\DeclareMathAlphabet{\mathpzc}{OT1}{pzc}{m}{it}
\DeclareSymbolFont{symbols}{OMS}{cmsy}{m}{n}
\usepackage[T1]{fontenc}
\usepackage{mathabx}

\newcounter{main}

\newtheorem{theorem}{Theorem}[section]
\newtheorem{proposition}[theorem]{Proposition}
\newtheorem{lemma}[theorem]{Lemma}

\newtheorem{remark}{Remark}[section]
\newtheorem{definition}{Definition}[section]
\newtheorem{maintheorem}{Theorem}

\newcommand{\blanksquare}{\,\,\,$\sqcup\!\!\!\!\sqcap$}

\newcounter{example}
\newenvironment{example}%
{{\stepcounter{example}}{\flushleft {\bf Example \arabic{example}:}}}%
{\par}

\newcommand{\tic}{\mathcal{T}_{\!I\!C}}
\newcommand{\sic}{\mathcal{S}_{\!I\!C}}
\newcommand{\gic}{\mathcal{G}_{\!I\!C}}

 \DeclareMathOperator{\esssup}{ess\,sup}
  \DeclareMathOperator{\GL}{GL}
    \DeclareMathOperator{\SL}{SL}
      \DeclareMathOperator{\Sp}{Sp}
   \DeclareMathOperator{\SO}{SO}
\def \dim{{\rm dim}}

\title[Fine properties of $L^p$-cocycles]{Fine properties of $L^p$-cocycles which allow abundance of simple and trivial spectrum}

\author[M. Bessa]{M\'{a}rio Bessa}

\author[H. Vilarinho]{Helder Vilarinho}

\begin{document}

\begin{abstract}
In this paper we generalize ~\cite{AC0} and prove that the class of accessible and saddle-con\-ser\-va\-tive cocycles (a wide class which includes cocycles evolving in $\GL(d,\mathbb{R})$, $\SL(d,\mathbb{R})$ and $\Sp(d,\mathbb{R})$) $L^p$-densely have a simple spectrum. We also generalize ~\cite{AC0,AB} and prove that for an $L^p$-residual subset of accessible cocycles we have a one-point spectrum, by using a different approach of the one given in ~\cite{AC0}. Finally, we show that the linear differential system versions of previous results also hold and give some applications.
\end{abstract}

\maketitle

\bigskip

\emph{\textbf{keywords:}} Linear cocycles; Linear differential systems; multiplicative ergodic theorem; Lyapunov exponents.\\

\section{Motivations and overview}

The question on knowing the asymptotic growth of the norm of the powers of a given matrix is a well-known exercise of linear algebra. Its Lyapunov spectrum, in terms of limit exponential behavior, which is defined by the Lyapunov exponents (i.e. logarithms of the eigenvalues) and eigendirections, are completely determined by using standard linear algebraic computations. Besides, the stability demeanor, when allowing perturbations, is a 	
fairly understood subject (see e.g. ~\cite{Ka}).  However, another question which is substantially harder, intends to understand the spectral properties of a given product of a collection (finite or infinite) of matrices and its stability. It is easy to see that even if we have only two matrices the spectrum can change drastically by a small change on the initial elements. We can think for instance, in combining a $2\times 2$ diagonal matrix different from the identity and also the identity matrix. The problem is reduced to the one described above, yet a small perturbation on the identity causes a substantial change in the final result, depending if we choose to keep it as a diagonal matrix or else we decide to input some rotational behavior.

In very general terms, there are mostly two ways of contextualize products of matrices: within the \emph{random} framework or else within the \emph{deterministic} one. In this paper we follow the deterministic viewpoint on which the deterministic behavior is established once we fix a map $T$ in a closed manifold $X$, an ``automatic generator matrices'' defined by a map $A$ from $X$ into a Lie subgroup of $\GL(d,\mathbb{R})$ and a mode of relating $T$ with $A$ (see \S\ref{cocycles} for full details). These objects are part of the language of the so-called \emph{linear cocycles} (see \cite[\S2 and \S3]{BP}). The existence of the previous mentioned objects like eigendirections and Lyapunov exponents are guaranteed once we have a $T$-invariant measure on $X$ and an integrability condition on $A$ (cf.~\cite{O}).

Choosing the accuracy on which we measure the size of a perturbation of the initial system will be crucial to answer the question of knowing the changes produced in the Lyapunov spectrum.

The goal of finding non-zero Lyapunov exponents is an old quest dating back to early 1980's and the work of Cornelis and Wojtkowski~\cite{CW}. About twenty years ago Knill~\cite{Kn} proved that non-zero Lyapunov exponents are a $C^0$-dense phenomena within bounded $\SL(2,\mathbb{R})$ cocycles. A  much sharper update was developed by Bochi ~\cite{B} taking into account the pioneering ideas of Ma\~n\'e~\cite{M1,M2} on rotation solutions (see also ~\cite{No}). Bochi observed that, from the more accurate $C^0$-generic point of view, we have the coexistence of strata on the manifold displaying positive Lyapunov exponents and hyperbolic behavior with other strata where zero Lyapunov exponents appeared (see also ~\cite{BV2} for generalizations). Observe that Cong \cite{Con2} improved the previous result for \emph{bounded} cocycles obtaining that a generic {bounded} $\SL(2,\mathbb{R})$-cocycle is uniformly hyperbolic, i.e., has a fibered exponential separateness. As far as we know, the best result on the abundance of simple spectrum (i.e. all Lyapunov exponents are different), on a quite large scope of topologies and on the two dimensional case, is given by a recent result of Avila (see~\cite{Av}).

From the continuous-time viewpoint we have the linear differential systems or skew-product flows which are, in general, morphisms of vector bundles covering a flow. As a quintessential example, we consider a dynamics given by a smooth flow, and in this case the morphism corresponds to the action of the tangent flow in the tangent bundle. These systems are the flow counterpart of the discrete cocycles, i.e., the $d$-dimensional ($d\geq 2$) \emph{linear
differential systems} over continuous $\mu$-invariant flows in compact Hausdorff
spaces $X$, where $\mu$ is a Borel regular measure. Linear differential systems are equipped  with a dynamics in the base $X$ given by a continuous flow $\varphi^{t}:X\rightarrow{X}$,
 a dynamics in the $d$-dimensional tangent bundle, given by a linear cocycle
$\Phi^{t}\colon X\rightarrow{\GL(d,\mathbb{R})}$ with time $t$ evolving on $\mathbb{R}$, and a certain relation between them (see \S\ref{LDS} for more details). This continuous-time case is somehow different from its discrete counterpart. For the $L^p$-denseness results we recall the statement in \cite{AC0} \emph{``...the
results of this paper (with some appropriate changes) can be applied to the continuous-time
case as well''}. In \S\ref{cc} we expose in detail those \emph{``appropriate changes''} pointed by Arnold and Cong. Moreover, we give the continuous-time version for our strategy in order to obtain the $L^p$-residuality of the one-point spectrum. We stress that any perturbation must be performed upon a given differential equation.

With respect to the continuous-time versions, in~\cite{Be,Be2}, it was proved the Ma\~{n}\'{e}-Bochi-Viana theorem for linear differential systems. We notice out that several particular examples of genericity of hyperbolicity (exponential dichotomy) in $C^{0}$-topology on the torus were already explored by Fabbri~\cite{F} and by Fabbri and Johnson~\cite{FJ1}. Some approaches have been proposed for determining the positivity of Lyapunov exponents for linear differential systems (see~\cite{F2,FJ2}). This last result follows from the paper of Kotani~\cite{Ko}. We suggest~\cite{FJZ} for a quite complete survey about these issues.

It is pretty clear that, for both discrete and continuous-time case, there exists lots of subtleties on this subject: the choice on the topology, the choice of $A$ being bounded or continuous, the choice of whether we take the dense or the generic viewpoint. The strengthening of this thesis can be pushed forward by recalling that, by one hand, Arnold-Cong~\cite{AC0} and Arbieto-Bochi~\cite{AB} proved that, for feeble topologies like the $L^p$-topology (see ~\S\ref{topologies}), generic cocycles have zero Lyapunov exponents. On the other hand, Viana~\cite{V} proved that for stronger topologies the positive Lyapunov exponents are prevalent (see also \cite{Av}). In between we have the Bochi-Ma\~n\'e dichotomy.

If we scrutinize carefully the Arnold and Cong strategy carried out in \cite{AC0} to obtain simple spectrum we observe that, besides the idiosyncrasy of the $L^p$-topology which allows large uniform-norm perturbations by making small $L^p$-perturbations, they used strongly two properties of the group of matrices, one from a \emph{topological} and the other from a \emph{geometric} nature:

\medskip

\begin{enumerate}
\item {\bf Topological Condition:} first they needed to commingle any different directions in the fiber space which they called \emph{``the turning solution method of Millionshchikov''} (see \cite{Mi}), and;
\item {\bf Geometric Condition:} second, they input a small expansion in the predefined direction on which the Lyapunov exponent should grow combined with a balanced contraction to give a volume invariance.
\end{enumerate}

\medskip

In the present paper we considered two abstract properties of subgroups of matrices which reflect (1) and (2) above and following the insight from the $L^p$-topology. In brief terms, the property (1) is called \emph{accessibility} and was already considered in ~\cite{BV2} (see also a related definition in  ~\cite{N}) and the property (2) is called \emph{saddle-conservativeness}. Once we formulate the results taking into account these two properties we derive easily that the theorems in the vein of those in \cite{AC0,AB} hold for the most important families of matrices, like, e.g., $\GL(d,\mathbb{R})$, $\SL(d,\mathbb{R})$ and $\Sp(2d,\mathbb{R})$. This was the strongest motivation for having opted for this abstract approach.

Another aspect that may raise some doubts to the reader and we intend to clarify at once was our choice not to follow the strategy of Arnold and Cong \cite{AC0} when we try to achieve the $L^p$-denseness of the one-point spectrum. In fact, we opt to develop the argument first used in ~\cite{BV2} allowing us to waive the ergodic hypothesis and deal with dynamical cocycles, and also allowing the approach to the infinite dimensional case. As an application, in \S\ref{dynamical} we apply our results to the dynamical cocycle given by the derivative of an area-preserving  diffeomorphism and endowed with the $L^p$-norm. In \S\ref{infinite}, we point out that for discrete $L^p$ cocycles evolving on compact operators of infinite dimension (cf. ~\cite{BeC}) the one-point spectrum is prevalent.

In the following table we consider an abbreviated summary of the prevalence of the different spectrums with respect to both discrete and continuous-time systems and also considering different type of topologies. \bigskip

\begin{center}
\small{\begin{tabular}{|c|c|c|c|}
	\hline
& {\tiny   \textbf{$L^p$-topology}} & {\tiny \textbf{$C^0$-topology} }  &   {\tiny \textbf{$C^{r+\alpha}$-topology} ($r\geq 0$, $\alpha>0$)}  \\
	\hline
 {\tiny \textbf{maps}}  & {\tiny  o.p.s. (\cite{AC0,AB}; Theorems~\ref{ops},  \ref{simple}, \ref{dc} and \ref{BeCLp}})   & {\tiny o.p.s. vs hyperbolicity (\cite{B, BV2}}) & {\tiny hyperbolicity (\cite{V,Av,BVar}})\\
	\hline
{\tiny \textbf{flows}} & {\tiny  o.p.s.  (Theorems~\ref{ops2} and \ref{simple2}})  & {\tiny o.p.s. vs hyperbolicity (\cite{FJ2,Be,Be2}}) & {\tiny hyperbolicity (\cite{BVar}})\\
  \hline
\end{tabular}}
\end{center}

\bigskip

This paper is organized as follows: in \S\ref{discrete}
we concern to discrete-time cocycles, where we establish the existence of an $L^p$-residual subset of the accessible cocycles
with one-point spectrum (Theorem  \ref{ops}) and the $L^p$-denseness of saddle-conservative accessible cocycles having simple
spectrum (Theorem~\ref{simple}). In \S\ref{cc} we treat with the continuous-time results. We state the existence of an
$L^p$-residual subset of the accessible linear differential systems with one-point spectrum (Theorem~\ref{ops2}) and the
$L^p$-denseness of saddle-conservative accessible linear differential systems having simple spectrum (Theorem~ \ref{simple2}).
Finally, in \S\ref{app} we apply our results to the dynamical cocycles given by the derivative of  area-preserving
diffeomorphims, and to discrete cocycles evolving on compact operators of infinite dimension.

\section{The discrete-time case}\label{discrete}

\subsection{Definitions and statement of the results}\label{discrete results} \subsubsection{Cocycles and Lyapunov exponents}\label{cocycles}

Let $X$ be a compact Hausdorff space, $\mu$ a Borel regular
non-atomic probability measure and $T:X\to X$ be an automorphism preserving $\mu$. Consider the set $\mathcal G$ of the ($\mu$ mod 0
equivalence classes of) measurable maps
$A:X\to \GL(d,\mathbb R)$, $d \geq2$, endowed with its Borel
$\sigma$-algebra. The Euclidean space $\mathbb{R}^d$ is endowed with the canonic inner product. Each map $A$ generates a linear cocycle
$$
\begin{array}{cccc}
F_A: & X\times\mathbb R^d & \longrightarrow & X\times\mathbb R^d \\
& (x,v) & \longmapsto & (T(x),A(x) v),
\end{array}
$$
over the dynamical system $T:X\to X$.
We set
$$A^n(x):= A(T^{n-1}(x))\cdots A(x)$$ for the composition of the maps $A(T^{n-1}(x))$ 	
up to $A(x)$ and, if
$T$ is invertible, $$A^{-n}:= A^{-1}(T^{-n}(x))\cdots
A^{-1}(T^{-1}(x)).$$ As usual, we consider $A^0:=\textrm{Id}$ where $\textrm{Id}$ stands for the $d\times d$ identity matrix.
By an abuse of language we will often identify $F_A$ and $A$.  Let $\|\cdot\|$ be an operator norm on the set $d\times d$ matrices with real entries. Consider the subset $\mathcal{G}_{\!I\!C}$ of $\mathcal G$ of all maps $A\in\mathcal G$ satisfying
the following \emph{integrability condition}:
\begin{equation*}
\int_X\log^+\|A^{\pm1}(x)\|\,d\mu<\infty,
\end{equation*}
 where
$\log^{+}(y)=\text{max}\,\{0,\log(y)\}$.
The multiplicative ergodic theorem of Oseledets \cite{O} ensures that the Lyapunov exponents
$\lambda_1(A,x)\geq\ldots\geq\lambda_d(A,x)$ of the \emph{integrable} cocycle $A\in\mathcal{G}_{\!I\!C}$ are defined for almost every point $x$.
If $T$ is ergodic, these functions are constant almost everywhere, as the possible values for the limits
$$\lim_{n\to\pm\infty}\frac1n\log \|A^n(x)v\|,$$
for $\mu$ almost every ($\mu$-a.e.) $x\in X$ and all $v\in\mathbb
R^d\setminus\{0\}$. We say that $A\in\mathcal{G}_{\!I\!C}$ has \emph{one-point (Lyapunov) spectrum} if all Lyapunov
exponents are equal. If, in addition, we include that the cocycle $A$ takes values in
$\SL(d,\mathbb R)$ then $A$ has one-point spectrum if and only if all
Lyapunov exponents are zero. On the other hand,  we say that $A\in\mathcal{G}_{\!I\!C}$ has \emph{simple (Lyapunov) spectrum} if all Lyapunov
exponents are different.

\subsubsection{A topology on cocycles}\label{topologies}
Let us endow $\mathcal G$ with an $L^p$-like topology as in \cite{AC0}. For $A, B\in\mathcal G$ and $1\leq p \leq \infty$ set
\begin{equation*}
\|A\|_p:=\left\{\begin{array}{lll}
                 \Big(\displaystyle\int_X \|A(x)\|^p\, d\mu\Big)^{1/p}, & &\textrm{if}\, 1\leq p < \infty \\
                 \displaystyle\esssup_{x\in X}\|A(x)\|, & & \textrm{if}\, p=\infty  \\
               \end{array}
\right.,
\end{equation*}
  and $$\Delta_p(A,B):=\|A-B\|_p+\|A^{-1}-B^{-1}\|_p.$$ We define now $$ d_p(A,B):=\frac{\Delta_p(A,B)}{1+\Delta_p(A,B)},$$
where $d_p(A,B)=1$ if $\Delta_p(A,B)=\infty$. According \cite{AC0}, $(\mathcal G,d_p)$, and hence $(\mathcal{G}_{\!I\!C},d_p)$ is a complete metric space.
\begin{remark}
It follows from the definition of the metric and from H\"older inequality (see e.g. \cite{Ru}) that, for all $A, B\in\mathcal G$ and $1\leq p\leq q\leq\infty$, we have $d_p(A,B)\leq d_q(A,B)$.
\end{remark}

\subsubsection{Families of cocycles}
We are interested in classes of maps $A$ taking values in specific subgroups of $\GL(d,\mathbb R)$.
In the greater generality we consider subgroups that satisfy an accessability type condition.

\begin{definition}
 We call $\mathcal{S}\subseteq \GL(d,\mathbb R)$ \textbf{accessible} if it is a non-empty closed subgroup of $\GL(d,\mathbb R)$ which acts transitively in the projective space $\mathbb RP^{d-1}$, that is, given $u,v\in \mathbb RP^{d-1}$, there is $R\in\mathcal{S}$ such that $R\,u=v$.
\end{definition}

\begin{example}
The subgroups $\GL(d,\mathbb R)$, $\SL(d,\mathbb R)$, $\Sp(2q,\mathbb R)$, as well $\GL(d,\mathbb C)$ and $\SL(d,\mathbb C)$, are accessible.
\end{example}

\begin{remark}
In \cite[Definition 1.2]{BV2} the authors introduced a slightly different notion of accessability. See \cite[Lemma 5.12]{BV2} for a relation between those concepts.
\end{remark}

Next result shows that accessibility allows us to reach anywhere within the projective space and acting on elements of the group.

\begin{lemma}\label{rot}
Let $\mathcal{S}$ be an accessible subgroup of $\GL(d,\mathbb R)$. There exists $K>0$ such that, for all $u,v\in\mathbb RP^{d-1}$,
there is $R_{u,v}\in \mathcal{S}$, with $\|R_{u,v}^{\pm1}\|\leq K$, such
that $R_{u,v} u=v$.
\end{lemma}

\begin{proof}
Fix some $\epsilon>0$ and let $0<\delta<\epsilon$ be such that if $R_1,R_2\in U_\delta:=\{R\in \mathcal{S}\colon \|R\|<\delta\}$, then $R_2R_1^{-1}\in U_{\epsilon}$. The hypothesis over $\mathcal S$ imply
that for any $w\in\mathbb RP^{d-1}$, the evaluation map $w\colon \mathcal{S}\to\mathbb RP^{d-1}$ given by $A\mapsto A(w)$ is open, so that $U_\delta(w):=\{Rw: R\in U_\delta\}$ is an open subset of
$\mathbb RP^{d-1}$. Due to the compactness of the projective space one can write $$\mathbb RP^{d-1}=U_\delta(w_1)\cup\cdots\cup U_\delta(w_{m}),$$
for some $m\geq1$. Let  $u,v\in\mathbb RP^{d-1}$ be given with $u\in U_\delta(w_{i})$ and $v\in U_\delta(w_{j})$ for some $i,j\in\{1,\ldots,m\}$. Let $R_u, R_v\in U_\delta$ be such that $R_u w_i=u, R_v w_j=v$. There exist $1\leq k\leq m$
and $\{R_{i}\}_{i\leq k}$, with $\|R_i\|<\epsilon$, such that $R_k\cdots R_1 w_i=w_j$. Then $R_{u,v} u:=R_vR_{k}\cdots R_{1}R_u^{-1} u=v$ and $\|R_{u,v}\|\leq (m+2)\epsilon$. We just have to consider $K=(m+2)\epsilon$.
\end{proof}

The next definition stresses the possibility of implementing some expansion in a given direction and simultaneously compensate with a contraction so that, ultimately, it preserves the volume. We note that the expansion will be used in the sequel when we want to enlarge a certain Lyapunov exponent under a small perturbation.

\begin{definition}\label{SC}
Let $\mathcal{S}$ be a closed subgroup of $\GL(d,\mathbb R)$. We call $\mathcal{S}$ \textbf{saddle-conservative} if given any direction $e\in \mathbb{R}^d$ and $\delta>0$ there exists $A_\delta\in \mathcal{S}$ such that:
\begin{enumerate}
\item $A_\delta \in \SL(d,\mathbb R)$ and
\item $A_\delta e =(1+\delta)e$.
\end{enumerate}

\end{definition}

\begin{example}
The groups $\GL(d,\mathbb R)$, $\SL(d,\mathbb R)$, $\Sp(2q,\mathbb R)$, as well as $\GL(d,\mathbb C)$ and $\SL(d,\mathbb C)$, display the saddle-conservative property. The special orthogonal group $\SO(d,\mathbb{R})$ is not saddle-conservative because, despite the fact that displays condition (1) it fails condition (2).
\end{example}

\subsubsection{Statement of the results}
Denote by $\tic$ an accessible subgroup of $\mathcal{G}_{\!I\!C}$ and by $\mathcal{S}_{\!I\!C}$ a saddle-conservative closed subgroup of $\mathcal{T}_{\!I\!C}$. We present now our first result:

\begin{maintheorem}\label{ops}
There exists an $L^p$-residual subset $\mathcal R\in \mathcal{T}_{\!I\!C}$,  $1\leq p < \infty$, such that any $B\in \mathcal{R}$ has one-point spectrum.
 \end{maintheorem}

Once we obtain an $L^p$-residual where one-point spectrum prevails we ask if it is possible to find $L^p$-open subsets where all Lyapunov exponents are equal. Clearly, this question is interesting if we exclude certain contexts where the problem becomes easy to solve. That is the case when we deal with cocycles evolving on isometry subgroups (like $\SO(d,\mathbb{R})$) or with cocycles evolving on compact subgroups, where we surely have one-point spectrum.

In order to reach simple spectrum similarly to \cite[Theorem 4.4]{AC0} we need to deal with subgroups displaying additional features. We this in mind we obtain:
\begin{maintheorem}\label{simple}
Let $T:X\to X$ be ergodic. For any $A\in \mathcal{S}_{\!I\!C}$,  $1\leq p < \infty$, and $\epsilon>0$ there exists $B\in \mathcal{S}_{\!I\!C}$, with $d_p(A,B)<\epsilon$ and $B$ has simple Lyapunov spectrum.
 \end{maintheorem}

 Theorems \ref{ops} and \ref{simple} are respectively  proved in \S\ref{proof ops} and \S\ref{proof simple}.

\subsection{One-point spectrum is $L^p$-residual}\label{proof ops}
The core argument in order to obtain a residual subset is by proving that a certain function related with the Lyapunov exponents of the cocycles is upper semicontinuous and that the continuity points are those with one-point spectrum. Once we have this proved we use the fact that the set of points of continuity of an upper semicontinuous function is always a residual subset (see e.g. \cite{K}). It was this idea that led Arbieto and Bochi \cite{AB} to improve the $L^p$-denseness result of Arnold and Cong \cite{AC0} for one-point spectrum cocycles to an $L^p$-residual grade.
In ~\cite[Theorem 4.5]{AC0} the one-point spectrum $L^p$-prevalence among cocycles evolving on $\GL(d,\mathbb R)$ is proved. Notwithstanding we believe that the proof of Arnold and Cong can be readapted to accessible cocycles, here we obtain the proof of this result following a different approach by reformulating the arguments developed in ~\cite{BV2}, where it was presented a strategy for equalizing the Lyapunov exponents with small perturbations in the delicate $C^0$ topology, and for a quite general classe of cocycles. One of the main purpose is to avoid the ergodicity condition on the dynamics over the base $T:X\to X$, which will be useful in some applications of our main results for discrete-time cocycles to dynamical cocycles (see \S\ref{dynamical}).

In this section we start by recalling the $L^p$-upper semicontinuity of the entropy function from Arbieto and Bochi ~\cite{AB}, and some elementary facts on exterior power and their relation to Lyapunov exponents. We revisit then the strategy of Bochi and Viana ~\cite{BV2} and give the (simplified) versions of some results adapted to our $L^p$ setting. We finish the section with the proof of Theorem~\ref{ops}. We inform the reader that our notation differs slightly from that of ~\cite{BV2}.

\subsubsection{The upper semicontinuity of the entropy function}\label{ArBo}
For $k=1,\ldots,d$ and $A\in\mathcal G_{\!I\!C} $ let $$\hat{\lambda}_k(A,x):= \lambda_1(A,x)+\ldots+\lambda_k(A,x)\quad\textrm{and}\quad\Lambda_k(A):=\int_X \hat\lambda_k(A,x)\,d\mu.$$

It was proved in ~\cite{AB} that $ A\mapsto \Lambda_k(A)$ is upper semicontinuous for all $k=1,\ldots,d$, with respect to the $L^p$-like topology, that is, for every $1\leq p \leq \infty$, $A\in\gic$ and $\varepsilon>0$ there exists $0<\delta<1$ such that, if $d_p(A,B)<\delta$ then $\Lambda_k(B)\leq \Lambda_k(A)+\varepsilon$. Moreover, $\Lambda_d$ is continuous on $\mathcal G_{IC}$. In particular, those results hold on the restriction of $\Lambda_k$ to the subsets $\mathcal S_{\!I\!C}$ and $\mathcal T_{\!I\!C}$ of $\mathcal G_{\!I\!C}$.

\subsubsection{Exterior powers}
The language of multilinear algebra is much appropriate when we want deal with several Lyapunov exponents (say $k$) for a cocycle $A$ by considering consider the dual problem of studying the upper Lyapunov exponent of the $k^{th}$ exterior product of $A$. Let us recall now some basic definitions. For details on multilinear algebra of operators see Arnold's book~\cite{A}.

The \emph{$k^{th}$ exterior power} of $\mathbb{R}^d$, denoted by $\wedge^{k}(\mathbb{R}^d)$, is also a vector space which satisfies $\dim(\wedge^{k}(\mathbb{R}^d))=(_{k}^{d})$. Given an orthonormal basis $\{e_{j}\}_{j=1}^d$ of $\mathbb{R}^d$, the family of exterior products $e_{j_{1}}\wedge e_{j_{2}}\wedge\ldots\wedge e_{j_{n}}$ for $j_{1}<\ldots<j_{k}$, with $j_{\alpha}\in \{1,\ldots,d\}$, constitutes an orthonormal basis of $\wedge^{k}(\mathbb{R}^d)$. Given a linear operator $A\colon \mathbb{R}^d\rightarrow \mathbb{R}^d$ we define the operator $\wedge^{k}(A)$, acting on the $k$-vector $u_{1}\wedge\ldots\wedge u_{k}$, by
$$\begin{array}{cccc}
\wedge^{k}(A)\colon & \wedge^{k}(\mathbb{R}^d) & \longrightarrow  & \wedge^{k}(\mathbb{R}^d) \\
& u_{1}\wedge\ldots\wedge u_{k} & \longmapsto  & A(u_{1})\wedge\ldots\wedge A(u_{k}).
\end{array}$$

As we already said, this operator will be very useful to prove our results since we can recover the spectrum and splitting information of the dynamics of $\wedge^{k}(A^n)$ from the one obtained by applying Oseledets' theorem to $A^n$. This information will be for the same full measure set and with this approach we deduce our results. Next, we present the multiplicative ergodic theorem for exterior power (for a proof see~\cite[Theorem 5.3.1]{A}).

\begin{lemma}\label{arnauld}
The Lyapunov exponents $\lambda_{i}^{\wedge k}(x)$ for $i\in\left\{1,\ldots,\left(_{k}^{d}\right)\right\}$, repeated
 with multiplicity, of the $k^{th}$ exterior product operator $\wedge^{k}(A)$ at $x$ are the following numbers given by the sums of the Lyapunov exponents of $A$ at $x$:
$$\sum_{j=1}^{k}\lambda_{i_{j}}(x), \text { where }1\leq i_{1}< \ldots<i_{k}\leq d.$$
This nondecreasing sequence starts with $\lambda_{1}^{\wedge k}(x)=\lambda_{1}(x)+\lambda_{2}(x)+\ldots+\lambda_{k}(x)$ and ends with $\lambda_{q(k)}^{\wedge k}(x)=\lambda_{d+1-k}(x)+\lambda_{d+2-k}(x)+\ldots+\lambda_{d}(x)$. Moreover, the splitting of $\wedge ^{k}(\mathbb{R}_{x}^{d}(i))$ for $0\leq i\leq q(k)$ (of $\wedge^{k}(A)$) associated to $\lambda_{i}^{\wedge k}(x)$ can be obtained from the splitting $\mathbb{R}_{x}^{d}(i)$ (of $A$) as follows; take an Oseledets basis $\{e_{1}(x),\ldots,e_{d}(x)\}$ of  $\mathbb{R}_{x}^{d}$ such that $e_{i}(x)\in E_{p}^{\ell}$ for $\dim(E_{x}^{1})+\ldots+\dim(E_{x}^{\ell-1})<i\leq \dim(E_{x}^{1})+\ldots+\dim(E_{x}^{\ell})$. Then, the Oseledets space is generated by the $k$-vectors:
$$e_{i_{1}}\wedge \ldots\wedge e_{i_{k}}\text { such that } 1\leq i_{1}<\ldots<i_{k}\leq d \text { and } \sum_{j=1}^{k}\lambda_{i_{j}}(x)=\lambda_{i}^{\wedge k}(x).$$
\end{lemma}

\subsubsection{Bochi-Viana's strategy revisited}\label{BVrevisited}

The following result is the $L^p$ version of ~\cite[Proposition 7.1]{BV2} which can be very simplified in the weak topologies that we are using. For the reader who is familiar with ~\cite{BV2}, we substantially simplify their proof because the \emph{third case} in the proof of ~\cite[Proposition 7.1]{BV2}, which deals with the concatenation of a large amount of small $C^0$-perturbations in the absent of a certain type of non-dominance, can be solved with a small sole $L^p$-perturbation. We can summarize by saying that the dominated splitting ceases to be an impediment of interchanging Oseledets directions by small $L^p$-perturbations.

\begin{lemma}\label{rot3}
Let be given $A\in \mathcal{T}_{\!I\!C}$, $1\leq p < \infty$, $\epsilon>0$, $y\in X$ a nonperiodic point and a nontrivial
splitting $\mathbb{R}^d= E\oplus F$ over $y$. Then, there exists $B\in \mathcal{T}_{\!I\!C}$, with $d_p(A,B)<\epsilon$, such
that $B(y) u = v$ for some nonzero vectors $u\in E$ and $v\in A(y) F$.
\end{lemma}

\begin{proof}
By Lemma~\ref{rot} there exists $K>0$ such that for $\hat{u},\hat{v}\in\mathbb RP^{d-1}$ with $u=\alpha\hat{u}\in E$ and $\hat{v}\in F$,
there is $R_{\hat u,\hat v}\in S$, with $\|R_{\hat u,\hat v}^{\pm1}\|\leq K$ such that $R_{\hat u,\hat v}\hat u=\hat v$. Let $V_\epsilon$
 be a small neighborhood of $y$ and we define the
following perturbation of $A$:
$$B(x)=\left\{\begin{array}{lll}
A(x),&&\textrm{if}\, x\notin V_\epsilon \\
\frac{1}{\|u\|}A(x) R_{\hat u,\hat v},&&\textrm{if}\, x\in V_\epsilon
\end{array}\right.,$$
It is clear that $d_p(A,B)<\epsilon$ if $V_\epsilon$ is sufficiently small. Moreover, $B(y) u\in A(y) F$.
 \end{proof}

The following proposition if the adapted version of ~\cite[Proposition 7.2]{BV2}.  Bearing in mind the aims we want to achieve we enumerate the main differences between them:

\begin{enumerate}
\item First of all we are using the $L^p$-like topology instead of the much more exigent $C^0$ topology. As a consequence, interchanging Oseledets' directions if a more simple task (compare Lemma~\ref{rot3} with ~\cite[Proposition 7.1]{BV2});
\item We observe that in \cite[Proposition 7.2]{BV2} is considered the subset $\Gamma^*_p(A,m)$ of points without an $m$-dominated splitting of index $k$. In our setting the dominated splitting is no more an obstruction to cause a decay on the Lyapunov exponents. For this reason we perform the perturbations in a full measure subset of $X$;
\item In \cite[Proposition 7.2]{BV2} the change of Oseledets directions are performed using several perturbations. On the contrary, due to Lemma~\ref{rot3}, in the present paper we only need one single perturbation which is done, more or less, on a half time iterate:
\end{enumerate}

\begin{proposition}\label{P1}
Consider $A\in \mathcal{T}_{\!I\!C}$, $\epsilon>0$, $\delta>0$ and $k\in\{1,\ldots,d-1\}$. There exists a measurable function  $N\colon X\rightarrow \mathbb{N}$ such that for $\mu$-a.e. $x\in X$ and every $n\geq N(x)$ there exists a linear map $B(T^{\frac{n}{2}}(x))$ (or $B(T^{\frac{n+1}{2}}(x))$ if $n$ is odd) such that:
$$\frac{1}{n}\log\|\wedge^k(A^{\frac{n}{2}-1}(T^{\frac{n}{2}+1}(x))\cdot B(T^{\frac{n}{2}}(x))\cdot A^{\frac{n}{2}}(x))\|\leq \delta+\frac{\hat\lambda_{k-1}(A,x)+\hat\lambda_{k+1}(A,x)}{2}.$$
\end{proposition}

We notice that $\|B(T^{\frac{n}{2}}(x))-A(T^{\frac{n}{2}}(x))\|$ can be, in general, very large. However, this is not a problem because the whole cocycle $B$ will be equal to $A$ outside a small neighborhood, thence $d_p(A,B)$ will be arbitrarily small for $1\leq p < \infty$. Moreover, let us note that the function $N$  above depends only on the a.e. asymptotic estimates given by Oseledets' theorem.

The following proposition is the adapted version of ~\cite[Proposition 7.3 and Lemma 7.4]{BV2} which fulfills the global picture of Proposition~\ref{P1}. We observe that its proof follows the same steps traversed in ~\cite{BV2}.

\begin{proposition}\label{P2}
Let be given $A\in \mathcal{T}_{\!I\!C}$, $1\leq p < \infty$, $\epsilon>0$, $\delta>0$ and $k\in\{1,\ldots,d-1\}$. There exists $B\in \mathcal{T}_{\!I\!C}$, with $d_p(A,B)<\epsilon$, such that
$$\Lambda_k(B)<\delta+\frac{\Lambda_{k-1}(A)+\Lambda_{k+1}(A)}{2}.$$

\end{proposition}

The end of the proof of Theorem~\ref{ops} is now a direct consequence of the arguments described in ~\cite[\S 4.3]{BV2} and ~\cite{AB} and the results proved above. We will present them now for the sake of completeness.

For each $k=1,\ldots,d-1$ we define the \emph{discontinuity jump} by:
$$J_k(A)=\int_X\frac{\lambda_k(A,x)-\lambda_{k+1}(A,x)}{2}\,d\mu.$$

The following result is Proposition~\ref{P2} rewritten.

\begin{proposition}\label{P3}
Given $A\in \mathcal{T}_{\!I\!C}$, $1\leq p < \infty$, $\epsilon>0$, $\delta>0$ and $k\in\{1,\ldots,d-1\}$, there exists $B\in \mathcal{T}_{\!I\!C}$, with $d_p(A,B)<\epsilon$, such that $$ \Lambda_k(B)<\delta-J_k(A)+ \Lambda_{k}(A).$$
\end{proposition}
We are now in conditions to finish the proof of Theorem~\ref{ops}:
\begin{proof}(of Theorem~\ref{ops})
Let $A\in \mathcal{T}_{\!I\!C}$ be a continuity point of the functions $\Lambda_k$ for all $k$. Then $J_k(A)=0$ for all $k$, i.e., $\lambda_k(A,x)=\lambda_{k+1}(A,x)$ for all $k$ and $\mu$-a.e. $x\in X$. Thence, the cocycle $A$ has one-point spectrum for $\mu$-a.e. $x\in X$. Finally, we recall that the set of continuity points of an upper semicontinuous function (cf. \S\ref{ArBo}) is a residual subset.
\end{proof}

\subsection{Simple spectrum is dense}\label{proof simple}
 In this section we prove Theorem~\ref{simple} by borrowing the simple spectrum part of ~\cite[\S 4]{AC0}. We start by establishing  in Lemma~\ref{split spec} the adaptation of ~\cite[Lemma 4.1]{AC0} with some adjustments that reflect our assumptions for the cocycle. This result allows us to split a one-point Lyapunov spectrum by an $L^p$-small perturbation of the cocycle. The remaining part of the proof of Theorem~\ref{simple} follows \emph{ipsis verbis}~\cite{AC0}. Throughout this section we will assume that $T:X\to X$ is ergodic.

\begin{lemma}\label{split spec}
Assume that $A\in \mathcal{S}_{\!I\!C}\subset \GL(d,\mathbb{R})$, with $d\geq2$, has one-point
spectrum. Then, for any small $\epsilon >0$ and $1\leq p <
\infty$, there exists $B\in\mathcal \mathcal{S}_{\!I\!C}$, with
$d_p(A,B)<\epsilon$, such that $B$ has at least two different
Lyapunov exponents.
\end{lemma}

\begin{proof}
Consider $M>1$ and a Borel subset $V\subset X$, such that $\mu(V)>0$,
 $V\cap T(V)=\emptyset$,
and $$\sup_{x\in V\cup T(V)}\|A^{\pm1}(x)\|\leq M,$$  and let
$$k(x):=\min\{ n\geq 1: T^{-n}(x)\in T(V)\}.$$ Fix a unitary vector $e\in
\mathbb RP^{d-1}$ and define the following vector which is a
normalized image under the cocycle $A$ of the vector $e$, in the
fiber corresponding to $x\in X$:
 $$v(x):=\left\{\begin{array}{lll}
e,&& \textrm{if}\, x\in T(V)\\
\frac{A^{k(x)}(T^{-k(x)}(x))e}{\|A^{k(x)}(T^{-k(x)}(x))e\|},&& \textrm{otherwise}
\end{array}\right., $$ and set $E(x)=\textrm{span}\{v(x)\}$.
For each $u\in\mathbb RP^{d-1}$ fix some $R_u:=R_{u,e}$ given by
Lemma \ref{rot}, with $\|R_{u}^{\pm1}\|\leq C_1$ and such that
$R_{u}u=e$. For $x\in V$ define $q(x)\in\mathbb RP^{d-1}$ given
by $$q(x)=\frac{A(x)v(x)}{\|A(x)v(x)\|}.$$ Define now the
following perturbation of $A$ in $V$:
$$C_1(x)=\left\{\begin{array}{lll}
A(x),&&\textrm{if}\, x\notin V  \,\textrm{or}\, q(x) =  e\\
R_{q(x)}A(x),&&\textrm{if}\, x\in V \,\textrm{and}\, q(x)\neq  e
\end{array}\right.$$
Since for $x\notin V$, $C_1(x)=A(x)$ and for $x\in V$ we have
\begin{equation*}\label{dist C1 A} \|C_1^{\pm 1}(x)-A^{\pm
1}(x)\|\leq \|A^{\pm1}(x)\|\cdot\|R_{q(x)}^{\pm 1}-\textrm{Id}\|,
\end{equation*}
it follows that $d_p(A,C_1)\leq \Delta_p(A,C_1)\leq
2MK\mu(V)^{1/p}$, which can be smaller than any small
$\epsilon>0$ just considering $V$ small enough $\mu$-measure. If
$C_1$ has two or more distinct Lyapunov exponents we take $B=C_1$
and we are done.

Let us consider now that $C_1$ has only one Lyapunov exponent
$\lambda_{C_1}$. Then it must be equal to the unique Lyapunov
exponent $\lambda_A$ for $A$ (and both have multiplicity $d$).
Indeed, since
$$\det A(x)=\det {C_1}(x)$$ for all $x\in X$, by the
multiplicative ergodic theorem we have
$$d .\lambda_{C_1}=\int \log |\det {C_1}(x)|\,d\mu=\int \log |\det
A(x)|\,d\mu=d.\lambda_A.$$

Now, let $\delta\in(0,1)$. Since our group has the saddle-conservative property,  we can find  $A_\delta\in \SL(d,\mathbb{R})$ such that $A_\delta e=(1+\delta)e$. We define now:
$$C_2(x)=\left\{\begin{array}{lll}
\textrm{Id}& & \textrm{if}\, x\notin T(V)\\
A_\delta(x) & & \textrm{if}\, x\in T(V)
\end{array}\right..$$
Finally, set $$D(x)=C_2(x) C_1(x).$$ Since, for all $x\in X$
\begin{equation*}\label{inv subspaces}
D(x)E(x)=C_1(x)E(x)=E(T(x)),
\end{equation*}
by Birkhoff's ergodic theorem we have
 for any $\delta>0$
\begin{align}\lambda(D,x,v(x))&:=\lim_{n\to\infty}\frac1n\log\|D^n(x)v(x)\|\nonumber\\&=\lim_{n\to\infty}\frac1n\log\|(1+\delta)^{\sum_{j=0}^{n-1}\mathbbm
l_V(T^j(x))}C_1^n(x)v(x)\|\nonumber\\
& = \lambda(C_1,x,v(x)) + \log(1+\delta)\mu(V)\label{lyap exponent
for Dflow}.
\end{align}
Let $\lambda_{D,1}>
\lambda_{D,2}>\ldots>\lambda_{D,r_\delta}$ be the distinct
Lyapunov exponents for $D$, with the corresponding multiplicities
$m_1, \ldots, m_{r_\delta}$. Since for all $x\in X$,
$$ \det D(x)=\det C_1(x)=\det
A(x)\label{dets},
$$
by the multiplicative ergodic theorem we also have
$$
\sum_{i=1}^{r_\delta}\lambda_{D,i}. m_{i}=d.\lambda_A.
$$ By \eqref{lyap exponent
for Dflow}, for any $\delta>0$ the cocycle $D$ has a Lyapunov exponent
equal to $\lambda_A+\log(1+\delta)\mu(V)$, so we must have
$r_\delta\geq2$. Moreover, for all $\delta>0$
\begin{align*}
\|D^{\pm 1}(x)-A^{\pm 1}(x)\|&\leq \|C_2^{\pm
1}(x)-\textrm{Id}\|\cdot\|A^{\pm 1}(x)\|\\
&\leq 2M \quad\textrm{for}\, x\in T(V),\\
\|D^{\pm 1}(x)-A^{\pm 1}(x)\|&\leq \|C_1^{\pm
1}(x)-A^{\pm 1}(x)\|\\
&\leq MK \quad\textrm{for}\, x\in V,\\
D(x)&=A(x)\quad \textrm{for}\, x\notin V\cup T(V),
\end{align*}
which implies
$$d_p(A,D)\leq \Delta_p(A,D)\leq 2(2+K)M\mu(V)^{1/p}.$$
For any given $\epsilon>0$ we can consider $V$ such that
$2(2+K)M\mu(V)^{1/p}<\epsilon$ and\ we just have to consider
$B=D$.

\end{proof}

In the next lemma \cite[Lemma 4.3]{AC0} we see that, under a small perturbation, we can change slighly the Lyapunov spectrum:
\begin{lemma}\label{lit change}
Assume that $A\in \mathcal{S}_{\!I\!C}$ has Lyapunov exponents $\lambda_{A,1}>\ldots>\lambda_{A,r}$ with multiplicities $m_1,\ldots,m_r$. Then, for any $\epsilon, \delta\in(0,1)$ and  Borel $U\subset X$ with $\mu(U)>0$, there exist $\epsilon_1\in(0,1)$ and $B\in \mathcal{S}_{\!I\!C}$, with $d_p(A,B)<\epsilon$, $1\leq p\leq \infty$, such that $B(x)=A(x)$, for $x\in X\setminus U$, and $B$ has Lyapunov exponents $\lambda_{A,1}+\epsilon_1\log(1+\delta)>\ldots>\lambda_{A_r}+\epsilon_1\log(1+\delta)$,  with multiplicities $m_1,\ldots,m_r$.
\end{lemma}

We are now in a position to argue for the proof of Theorem~\ref{simple}:

 \begin{proof}(of Theorem~\ref{simple})
 Let $\{E_1(x),\ldots,E_r(x)\}$ be the Oseledets splitting of $\mathbb{R}^d$ generated by $A\in\sic$ and let $\{A_1(x),\ldots,A_r(x)\}$ be the corresponding decomposition of $A(x)=\bigoplus_{i=1}^rA_i(x)$. The idea is to apply Lemma~\ref{split spec} and Lemma~\ref{lit change} (if necessary) on the sub-bundles $E_i$. We stress that the proofs of Lemmas~\ref{split spec} and~\ref{lit change} allow us to perturb the original cocycle on a set of small $\mu$-measure of our choice, and can be taken to each of the blocks $A_i$ separately, without influencing the other blocks. The procedure is to look if $\dim(E_1(x))\geq2$ and, in this case, apply Lemma~\ref{split spec} to split this sub-bundle by a perturbation $B_1'$ of $A_1$ with at least to different Lyapunov exponents and, if necessary, combine it with Lemma~\ref{lit change} to get $B_1\in\sic$, with $d_p(A,B_1)<\epsilon/d$ with at least $r+1$ distinct Lyapunov exponents in its spectrum. We continue this procedure and after at most $d-1$ steps we
obtain $B\in\sic$ with $d_p(A,B)<\epsilon$ and with simple spectrum.\end{proof}

\section{The continuous-time case}\label{cc}

\subsection{Definitions and statement of the results}\label{cont results}

\subsubsection{Linear differential systems and Lyapunov exponents}\label{LDS}

Let $X$ be a compact Hausdorff space, $\mu$ a Borel regular
measure and $\varphi^{t}:X\rightarrow{X}$ a one-parameter family
of continuous maps for which $\mu$ is $\varphi^{t}$-invariant.
A cocycle based on $\varphi^{t}$ is defined by a flow
$\Phi^{t}(x)$ differentiable on the time parameter
$t\in{\mathbb{R}}$, measurable on space-parameter $x\in{X}$, and
acting on $\GL(d,\mathbb{R})$.  Together they form the linear
skew-product flow:

$$
\begin{array}{cccc}
\Upsilon^{t}: & X\times{\mathbb{R}^{d}} & \longrightarrow & X\times{\mathbb{R}^{d}} \\
& (x,v) & \longmapsto & (\varphi^{t}(x),\Phi^{t}(x){v})
\end{array}
$$

The flow $\Phi^{t}$ satisfies the so-called \emph{cocycle identity}: $\Phi^{t+s}(x)=\Phi^{s}(\varphi^{t}(x)){\Phi^{t}(x)}$,
for all $t,s\in{\mathbb{R}}$ and $x\in{X}$. If we define a map $A\colon X\rightarrow{{\mathfrak {gl}}(d,\mathbb{R})}$ in
a point $x\in{X}$ by:
$$A(x)=\frac{d}{ds}\Phi^{s}(x)|_{s=0}$$
and along the orbit $\varphi^{t}(x)$ by:
\begin{equation}\label{lvi}
A(\varphi^{t}(x))=\frac{d}{ds}\Phi^{s}(x)|_{s=t} {[\Phi^{t}(x)]^{-1}},
\end{equation}
then $\Phi^{t}(x)$ will be the solution of the linear variational equation (or equation of first variations):
\begin{equation}\label{lve}
\frac{d}{ds}{u(x,s)|_{s=t}}=A(\varphi^{t}(x)) u(x,t),
\end{equation}
and $\Phi^{t}(x)$ is also called the \emph{fundamental matrix} or the \emph{matriciant}
of the system (\ref{lve}). Given a
cocycle $\Phi^{t}$ we can induce the associated \emph{infinitesimal generator} $A$ by
using~\eqref{lvi} and given $A$ we can recover the cocycle by
solving the linear variational equation~\eqref{lve}, from which we get
$\Phi_{A}^{t}$. In view of this, sometimes we refer for $A$ as a \emph{linear differential system}. Moreover, if  in addition, $A$ is continuous with respect to the space variable $x$, we call $A$ a \emph{continuous linear differential system}.

Several type of linear differential system are of interest, the ones with
invertible matriciants, for all $x\in X$ and $t\in \mathbb{R}$, denoted by $\mathfrak{gl}(d,\mathbb{R})$, the \emph{traceless} ones with volume-preserving matriciant, for all $x\in X$ and $t\in \mathbb{R}$, which we denote by
$\mathfrak{sl}(d,\mathbb{R})$, and also the systems with matriciant evolving in the symplectic group $\Sp(2d,\mathbb{R})$, denoted by
$\mathfrak{sp}(2d,\mathbb{R})$.

\medskip

\begin{example}
An illustrative example is the linear differential system associated to flows $X^t$ with $\|X(x)\|\not=0$, where $X(x)=\frac{d}{dt}X^t(x)|_{t=0}$ and $x\in X$.  In this case we have $\Phi^{t}(x)\in{\GL(d,\mathbb{R})}$, and so the infinitesimal generator, given by relation (\ref{lvi}),
 belongs to $\mathfrak{gl}(d,\mathbb{R})$. Another example is the linear differential system associated to incompressible flows $X^t$ where $\|X(x)\|=1$ for any $x\in X$.  In this case we have $\Phi^{t}(x)\in{\SL(d,\mathbb{R})}$, and so the infinitesimal generator belongs to $\mathfrak{sl}(d,\mathbb{R})$.
\end{example}

\medskip

Consider the subset $\mathscr{G}_{\!I\!C}$ of maps $A\colon X\rightarrow\mathfrak{gl}(d,\mathbb{R})$
belonging to $L^1(\mu)$ that is: $$\int_X \|A(x)\|\,d\mu<\infty.$$
For such infinitesimal generators there is a unique, up to indistinguishability, linear differential system $\Phi_A^t$ satisfying, for $\mu$-a.e. $x$,
\begin{equation}\label{solu}\Phi_A^t(x)=\textrm{Id}+\int_0^t A(\varphi^s(x))\Phi_A^s(x)\,ds.\end{equation}
In this conditions, the time-one solution satisfies
the  \emph{integrability condition}
\begin{equation*}\label{IC}
\int_X\log^+\|\Phi_A^{\pm1}(x)\|\,d\mu<\infty,
\end{equation*}
and,  consequently, Oseledets' theorem guarantees that for $\mu$-a.e.  $x\in X$, there exists a $\Phi_{A}^{t}$-invariant splitting called \emph{Oseledets' splitting} of the fiber $\mathbb{R}^{d}_{x}=E^{1}(x)\oplus \ldots \oplus E^{k(x)}(x)$ and real numbers called \emph{Lyapunov exponents} $\tilde{\lambda}_{1}(x)>\ldots>\tilde{\lambda}_{k(x)}(x)$, with $k(x)\leq d$, such that:
\begin{equation*}\label{limit}
\underset{t\rightarrow{\pm{\infty}}}{\lim}\frac{1}{t}\log{\|\Phi_{A}^{t}(x) v^{i}\|={\tilde\lambda}_{i}(x)},
\end{equation*}
for any $v^{i}\in{E^{i}(x)\setminus\{\vec{0}\}}$ and $i=1,\ldots,k(x)$. If we do not count the multiplicities, then we have $\lambda_{1}(x)\geq \lambda_{2}(x)\geq\ldots\geq\lambda_{d}(x)$. Moreover, given any of these subspaces $E^{i}$ and $E^{j}$, the angle between them along the orbit has subexponential growth, meaning that
\begin{equation*}\label{angle}
\lim_{t\rightarrow{\pm{\infty}}}\frac{1}{t}\log\sin(\measuredangle(E^{i}({\varphi^{t}(x)}),E^{j}({\varphi^{t}(x)})))=0.
\end{equation*}
If the flow $\varphi^{t}$ is ergodic, then the Lyapunov exponents and the dimensions of the associated subbundles
are $\mu$-a.e. constant. For this results on linear differential systems see \cite{A} (in particular, Example 3.4.15). See also ~\cite{JPS}.

As before, we say that $A\in\mathscr{G}_{\!I\!C}$ has \emph{one-point (Lyapunov) spectrum} if all Lyapunov
exponents are equal. If, moreover, the linear differential system $A$ takes values in
$\mathfrak{sl}(d,\mathbb R)$, then $A$ has one-point spectrum if and only if all
Lyapunov exponents are zero. On the other hand,  we say that $A\in\mathscr{G}_{\!I\!C}$ has \emph{simple (Lyapunov) spectrum} if all Lyapunov
exponents are different.

\subsubsection{Topologies on linear differential systems}\label{top}

Consider the set $\mathscr{G}$ of the measurable maps
$A:X\to \mathfrak{gl}(d,\mathbb R)$, $d \geq2$, endowed with its Borel
$\sigma$-algebra. For $A,B\in\mathscr{G}$ and $1\leq p\leq\infty$ set
\begin{equation*}
\|A\|_p:=\left\{\begin{array}{lll}
                 \Big(\displaystyle\int_X \|A(x)\|^p d\mu\Big)^{1/p}, & &\textrm{if}\, 1\leq p < \infty \\
                 \displaystyle\esssup_{x\in X}\|A(x)\|, & & \textrm{if}\, p=\infty  \\
               \end{array}
\right..
\end{equation*}
and
\begin{equation}\label{metric} d_p(A,B)=\frac{\|A-B\|_p}{1+\|A-B\|_p},\end{equation}
where $d_p(A,B)=1$ if $\|A-B\|_p=\infty$. Note that $A(x)\in \mathfrak{gl}(d,\mathbb R)$ do not need to be invertible.

As in the discrete-time setting, the equality \eqref{metric} defines a metric on the space of infinitesimal generators, which is complete with respect to this metric. We refer for the metric/norm/topology induced by \eqref{metric} has the \emph{$L^p$ infinitesimal generator metric/norm/topology}.

\begin{remark}\label{metric rel2}
It follows from the definition of the metric and from H\"older inequality that, for all $A, B\in\mathscr{G}$ and $1\leq p\leq q\leq\infty$, we have $d_p(A,B)\leq d_q(A,B)$.
\end{remark}

\begin{remark}\label{small dist implies ic}
If $A\in\mathscr{G}_{\!I\!C}$ and $B\in\mathscr{G}$ with $d_p(A,B) < 1$, $1\leq p\leq\infty$, then $B\in\mathscr{G}_{\!I\!C}$; see~\cite{AC0}.
\end{remark}

\subsubsection{Families of linear differential systems}

Like we did in the discrete case we are interested in elements $A$ taking values in specific subgroups of $\mathfrak{gl}(d,\mathbb{R})$.
In the greater generality we consider subgroups that satisfy an accessibility condition:

\begin{definition}
We call a non-empty closed subalgebra $\mathscr{T}\subset \mathfrak{gl}(d,\mathbb R)$ \textbf{accessible} if its associated Lie subgroup acts transitively in the projective space $\mathbb RP^{d-1}$.
\end{definition}

\begin{example}
The subalgebras $\mathfrak{gl}(d,\mathbb R)$, $\mathfrak{sl}(d,\mathbb R)$, $\mathfrak{sp}(2d,\mathbb R)$ are accessible.
\end{example}

\begin{lemma}\label{rot2}
Let $\mathscr{T}$ be an accessible subalgebra of $\mathfrak{gl}(d,\mathbb R)$. Then,
there exists $K>0$ such that for all $u,v\in\mathbb RP^{d-1}$
there is $\{\mathfrak{R}_{u,v}(t)\}_{t\in[0,1]}\in \mathscr{T}$, with $\|\mathfrak{R}_{u,v}(t)\|\leq K$ such
that $\Phi_{\mathfrak{R}_{u,v}}^1 u=v$, where $\Phi^t_{\mathfrak{R}_{u,v}}$ is the solution of the linear variational equation $\dot{u}(t)=\mathfrak{R}_{u,v}(t)\cdot u(t)$.
\end{lemma}

\begin{proof}
The proof is analog to the one in Lemma~\ref{rot}. In order to comply the continuous-time formalization we just have to consider a smooth isotopy on $\mathscr{T}$ from the identity to the rotation $R_{u,v}$ (which sends the direction $u$ into the direction $v$) given by $\zeta(t)$, with $\zeta(t)=\textrm{Id}$ for $t\leq 0$ and $\zeta(t)=R_{u,v}$ for $t\geq 1$. We consider the linear variational equation
$$\dot{u}(t)=\left[\frac{d}{dt}\zeta(t)\cdot \zeta(t)^{-1}\right]\cdot u(t)$$ with initial condition $u(0)=\textrm{Id}$ and unique solution equal to $\zeta(t)$. Define $\mathfrak{R}_{u,v}(t)=\frac{d}{dt}\zeta(t)\cdot \zeta(t)^{-1}$. Clearly, $\mathfrak{R}_{u,v}(t)$ is bounded. Moreover, the solution of $\dot{u}(t)=\mathfrak{R}_{u,v}(t)\cdot u(t)$ defined by $\Phi^t_{\mathfrak{R}_{u,v}}$ is, such that,
$$\Phi^1_{\mathfrak{R}_{u,v}} u=\zeta(1) u=v.$$
\end{proof}

\begin{definition}
We say that  a closed Lie subalgebra  $\mathscr{S}\subseteq\mathfrak{gl}(d,\mathbb R)$ is \textbf{saddle-conservative} if its associated Lie subgroup is saddle-conservative in the sense of Definition~\ref{SC}.
\end{definition}

\begin{example}
Analogous to the discrete-time case we have that the Lie algebras $\mathfrak{gl}(d,\mathbb R)$, $\mathfrak{sl}(d,\mathbb R)$, $\mathfrak{sp}(2q,\mathbb R)$ display the saddle-conservative property. The orthogonal Lie algebra and the special orthogonal Lie algebra do not display the saddle-conservative property.
\end{example}

\medskip

Denote by $\mathscr{T}_{\!I\!C}\subset \mathscr{G}_{\!I\!C}$ the maps $A\colon X\rightarrow \mathscr{T}\subset\mathfrak{gl}(d,\mathbb R)$ where $\mathscr{T}$ is an accessible subalgebra. Denote by $\mathscr{S}_{\!I\!C}\subset\mathscr{T}_{\!I\!C}$ the maps $A\colon X\rightarrow \mathscr{S}\subset\mathscr{T}$ where $\mathscr{S}$ is a saddle-conservative accessible subalgebra.

\subsubsection{Conservative perturbations}

Considering the same notation as before we recall the Os\-tro\-grad\-sky-Jacobi-Liouville formula:
\begin{equation}\label{OJL}
\exp\left({\int_{0}^{t}\text{Tr}\,A(\varphi^{s}(x))\,ds}\right)=\det \Phi^{t}_A(x),
\end{equation}
where $\text{Tr}(A)$ denotes the trace of the matrix $A$.

Therefore, we may speak about conservative perturbations of systems
$A$ evolving in $\mathfrak{gl}(d,\mathbb{R})$ along the orbit
$\varphi^{t}(x)$ as $A+H$ where
$H(\varphi^{t}(x))\in \mathfrak{sl}(d,\mathbb{R})$. Denote by $\Phi_A^t$ the solution of (\ref{lve}) and by $\Phi_{A+H}^t$ the solution of the perturbed system:
\begin{equation*}\label{lve3}
\frac{d}{ds}{u(x,s)|_{s=t}}=[A(\varphi^{t}(x))+H(\varphi^{t}(x)]\cdot u(x,t),
\end{equation*}
By a direct application of formula (\ref{OJL}) we obtain
 \begin{eqnarray*}
 \det(\Phi^t_{A+H}(x))&=&exp\left({\int_{0}^{t}\text{Tr}A(\varphi^{s}(x))+\text{Tr}H(\varphi^{s}(x))\,ds}\right)\\
 &=&exp\left({\int_{0}^{t}\text{Tr}A(\varphi^{s}(x))\,ds}\right)\\
 &=&\det(\Phi^{t}_A(x)),
 \end{eqnarray*}
which allows us to conclude that the perturbation leaves the volume form invariant.

\subsubsection{Statement of the results}

We intend to obtain the continuous-time version of the discrete results treated in the first part of this paper. We start by establishing the existence of a $L^p$-residual of the accessible linear differential systems with one-point spectrum:
\begin{maintheorem}\label{ops2}
There exists an $L^p$-residual subset $\mathcal R\in \mathscr{T}_{\!I\!C}$,  $1\leq p < \infty$ such that, for any $B\in \mathcal{R}$ we have that $B$ has one-point spectrum.
 \end{maintheorem}
However, there are no $L^p$-open subsets of the saddle-conservative accessible linear differential systems, since the simple spectrum is a dense property:
\begin{maintheorem}\label{simple2}
For any $A\in \mathscr{S}_{\!I\!C}$,  $1\leq p < \infty$ over an ergodic flow and $\epsilon>0$, there exists $B\in \mathscr{S}_{\!I\!C}$, with $d_p(A,B)<\epsilon$ and $B$ has simple Lyapunov spectrum.
 \end{maintheorem}

\subsection{The Arbieto and Bochi theorem for linear differential systems}

Let us consider the following function where $\mathscr{L}$ is one of the subsets of linear differential systems $\mathscr{T}_{\!I\!C}$, $\mathscr{S}_{\!I\!C}$ or $\mathscr{G}_{\!I\!C}$:
\begin{equation*}\label{entropy}
\begin{array}{cccc}
\Lambda_{k}\colon &\mathscr{L} & \longrightarrow & [0,\infty) \\
& A & \longmapsto & \int_{X}\lambda_{1}(\wedge^{k}(A),x)\, d\mu.
\end{array}
\end{equation*}
With this function we compute the integrated \emph{largest} Lyapunov exponent of the $k^{th}$ exterior power operator.
Let us denote $\hat\lambda_{k}(A,x)=\lambda_{1}(A,x)+\ldots+\lambda_{k}(A,x)$. By using Lemma~\ref{arnauld} we conclude that  for $k=1,\ldots,d-1$ we have $\hat\lambda_{k}(A,x)=\lambda_{1}(\wedge^{k}(A),x)$ and therefore we obtain  $\Lambda_{k}(A)=\Lambda_{1}(\wedge^{k}(A))$.

In order to prove  that $\Lambda_k$ is an upper semicontinuous function if we endow  $ \mathscr{L}$ with the $L^p$ infinitesimal generator topology (Proposition \ref{upper sc}),
we give a preliminary result which allows us to control different solutions taking into account the closeness of the respective infinitesimal generators.

In what follows we use the same notation for the $L^1$-norm of the infinitesimal generators introduced in \S\ref{top} and for the usual $L^1$-norm $\|f\|_1$ of functions $f:X\to\mathbb R$,
given by $\int_X |f(x)|\,d\mu$.

\begin{lemma}\label{cont solu}
For $A,B\in\mathscr{G}_{\!I\!C}$ we have
$$\left\|\log^+\|\Phi_A^t(x)\|-\log^+\|\Phi_B^t(x)\|\right\|_1\leq t\|A-B\|_1,\,\,\text{for all}\,\,\, t\in\mathbb R^+.$$
\end{lemma}
\begin{proof}
From \eqref{solu}, Gronwall's lemma (see, e.g., \cite{A}) implies that, with $C=A,B$, for  $\mu$-a.e. $x\in X$ and for all $t\in\mathbb R^+$ we have
$$\log^+\|\Phi_{C}^t(x)\|\leq\int_{0}^t\|C(\varphi^s(x))\|\,ds,$$
and, consequently,
\begin{align*}
\left|\log^+\|\Phi_A^t(x)\|-\log^+\|\Phi_B^t(x)\|\right|
&\leq \left|\int_0^t \|A(\varphi^s(x))\|-\|B(\varphi^s(x))\|\,ds\right|\\
&\leq \int_0^t \|A(\varphi^s(x))-B(\varphi^s(x))\|\,ds=:\alpha_t(x).
\end{align*}
By \cite[Lemma 2.2.5]{A} $\alpha_t(x)\in L^1(X)$, and by Tonelli-Fubini's theorem, the change of variables theorem and the $\varphi^s$-invariance of $\mu$, we have for all $t\in\mathbb R^+$
\begin{align*}\label{pseudocont}
\left\|\log^+\|\Phi_A^t(x)\|-\log^+\|\Phi_B^t(x)\|\right\|_1
&\leq\int_X \int_0^t \|A(\varphi^s(x))-B(\varphi^s(x))\|\,ds\,d\mu\\
&\leq \int_0^t \|A-B\|_1\,ds\\
&= t\|A-B\|_1.
\end{align*}
\end{proof}
Recall that, for any $A\in\mathscr{G}_{\!I\!C}$ we have
\begin{equation}\label{exp via inf}
\Lambda_k(A)=\underset{t\rightarrow{\pm{\infty}}}{\lim}\frac{1}{t}\int_X\log\|\wedge^k(\Phi_{A}^{t}(x))\|\,d\mu=\underset{n\in\mathbb{N}}{\inf}\,\frac{1}{n}\int_X\log\|\wedge^k(\Phi_{A}^{n}(x))\|\,d\mu.
\end{equation}

\begin{proposition}\label{upper sc}
For each $k=1,\ldots, d$, the function $\Lambda_k$ is upper semicontinuous when we endow $\mathscr{L}$ with the $L^p$ infinitesimal generator topology, $1\leq p\leq\infty$. Moreover, in these conditions $\Lambda_d$ is a continuous function.
\end{proposition}

\begin{proof}
Let $A\in\mathscr{G}_{\!I\!C}$, $k\in\{1,\ldots,d\}$ and $\epsilon>0$ be given. We start by assuming that
\begin{equation}\label{hatlambda geq 0}
\hat\lambda_k(A,x)\geq 0 ,\,\,\text{for}\,\, \mu\text{-a.e.} \,\,x\in X.
\end{equation}
By \eqref{exp via inf}, \eqref{hatlambda geq 0} and the subbaditive ergodic theorem, it is possible to find $N\in\mathbb{N}$ large enough in order to have
\begin{equation}\label{LBA}
\frac{1}{N}\int_X\log^+\|\wedge^k(\Phi_{A}^{N}(x))\|\,d\mu<\Lambda_k(A)+\frac\epsilon2.
\end{equation}
We will see that we can find $\delta>0$ such that for any $B$ satisfying $d_p(A,B)<\delta$ we have that $B\in \mathscr{G}_{\!I\!C}$ (this follows from Remarks \ref{metric rel2} and \ref{small dist implies ic}) and
$\Lambda_k(B)< \Lambda_k(A)+\epsilon$. Indeed, since $\|\wedge^k\Phi_{A,B}^N(x)\|\leq \|\Phi_{A,B}^N(x)\|^k$, from \eqref{exp via inf}, \eqref{LBA} and Lemma \ref{cont solu} we get
\begin{align*}
\Lambda_k(B)
&\leq \frac{1}{N}\int_X\log^+\|\wedge^k(\Phi_{B}^{N}(x))\|\,d\mu\\
&\leq \frac{1}{N}\int_X\log^+\|\wedge^k(\Phi_{A}^{N}(x))\|\,d\mu+ \frac{1}{N}\int_X\left|\log^+\|\wedge^k(\Phi_{B}^{N}(x))\|-\log^+\|\wedge^k(\Phi_{B}^{N}(x))\|\right|\,d\mu\\
&\leq\Lambda_k(A)+\frac\epsilon2+\frac{k}N N\|A-B\|_1.
\end{align*}
If $\delta<\epsilon/({2k+\epsilon})$ then $d_p(A,B)<\delta$ implies $\|A-B\|_1\leq\|A-B\|_p<\epsilon/({2k})$, and the result follows.

Let us prove now the general case. Again, let $A\in\mathscr{G}_{\!I\!C}$, $k\in\{1,\ldots,d\}$ and $\epsilon>0$ be given. For $\alpha>0$ we define the $\varphi^t$-invariant set $L_\alpha=\{x\in X: \hat\lambda_k(A,x)<-\alpha\}$. Consider $\alpha$ large enough such that
\begin{equation}\label{gen case 1}
k\int_{L_\alpha}\log^+\|\Phi_A^1(x)\|\,d\mu < \frac\epsilon8\,\,\,\, \text{and} \,\,\,\, \int_{L_\alpha} \hat\lambda_k(A,x)\,d\mu>-\frac\epsilon8.
\end{equation}
Set $\beta\geq \alpha>0$, denote by $\textrm{Id}$ the identity $d\times d$ matrix and define $\tilde A(x) = A(x)+\beta.\textrm{Id}$, $\tilde B(x) = B(x)+\beta.\textrm{Id}$. Then $\hat\lambda_k(\tilde A,x)=\hat\lambda_k(A,x)+\beta,$ which is greater or equal than zero for $x\in L_\alpha^C$.
Moreover, if $d_p(A,B)$ is sufficiently small then also is $d_p(\tilde A,\tilde B)$, and by the previous case we have \begin{equation*}\int_{L_\alpha^C}\hat\lambda(\tilde B,x)\,d\mu\leq \int_{L_\alpha^C}\hat\lambda(\tilde A,x)\,d\mu +\frac\epsilon2,
\end{equation*}
which implies
\begin{equation}\label{gen case 2}\int_{L_\alpha^C}\hat\lambda(B,x)\,d\mu\leq \int_{L_\alpha^C}\hat\lambda(A,x)\,d\mu +\frac\epsilon2.
\end{equation}

From Lemma \ref{cont solu}, if $d_p(A,B)$ is sufficiently small then $$\left\|\log^+\|\Phi_A^1(x)\|-\log^+\|\Phi_B^1(x)\|\right\|_1\leq \frac{\epsilon}{4k},$$
which, with \eqref{gen case 1} implies
\begin{eqnarray}
\int_{L_a}\hat\lambda(B,x)\,d\mu&=& \inf_n\frac1n\int_{L_a}\log^+\|\wedge^k\Phi_B^n(x)\|\,d\mu\nonumber\\
&\leq& k\int_{L_a}\log^+\|\Phi_B^1(x)\|\,d\mu\nonumber\\
&\leq&k\int_{L_a}\log^+\|\Phi_A^1(x)\|\,d\mu+k\int_{L_a}\left|\log^+\|\Phi_A^1(x)\|-\log^+\|\Phi_B^1(x)\|\right|\,d\mu\nonumber\\
&\leq&\int_{L_a} \hat\lambda_k(A,x)\,d\mu+\frac\epsilon2\label{gen case 4}.
\end{eqnarray}
The proof for this general case follows now from \eqref{gen case 2} and \eqref{gen case 4}. Finally, in order to prove the continuity of $\Lambda_d$ we just have to note that
$$A\mapsto\tilde\Lambda_k(A):=\int_X\lambda_{d-k+1}(A,x)+\cdots+\lambda_{d}(A,x)\,d\mu=-\Lambda_k(-A)$$
is lower semicontinuous for each $k=1,\ldots,d$, so that $\Lambda_d=\tilde\Lambda_d$ is continuous.
\end{proof}

\subsection{One-point spectrum is residual}\label{ops cont}

The proof of Theorem~\ref{ops2} is a straightforward application of the scheme described in \S\ref{BVrevisited} to prove Theorem~\ref{ops}. The only novelty is the perturbation toolbox which we will develop in the sequel (Lemma~\ref{rot3cont}). We consider the perturbations within the \emph{continuous} linear differential systems because the estimates are more easily established. Once we have a perturbation framework developed the proof of Theorem~\ref{ops2} will have a further simple additional step.

\begin{proof}(of Theorem~\ref{ops2})
Let $A\in \mathscr{T}_{\!I\!C}$ be a continuity point of the functions $\Lambda_k$, for all $k=1,...,d$, defined in Proposition~\ref{upper sc}, and with respect to the $L^p$-topology.
\begin{enumerate}
\item [Case 1:] $A$ is a continuous linear differential system. We proceed as in the proof of Theorem~\ref{ops} and use the perturbation Lemma~\ref{rot3cont} to mix Oseledets direction and so cause a decay of the Lyapunov exponents and finally we use Proposition~\ref{upper sc} to complete the argument.
\item [Case 2:] $A$ is not a continuous linear differential system. It follows from Lusin's theorem (see e.g. ~\cite[\S 2 and \S 3]{Ru}) that the continuous linear differential systems over flows on compact spaces $X$ and on manifolds like the Lie subgroups we are considering, are $L^p$-dense in the $L^p$ ones.
\end{enumerate}
Now, we take a sequence of continuous linear differential systems $A_n\in \mathscr{T}_{\!I\!C}$ converging to $A$ in the $L^p$-sense. Since $A$ is a continuity point we must have $\underset{n\rightarrow \infty}{\lim} \Lambda_k(A_n)=\Lambda_k(A)$.
Like we did in Proposition~\ref{P3}, but this time in the flow setting, given $\epsilon_n\rightarrow 0$ and $\delta>0$, there exists $B_n\in \mathscr{T}_{\!I\!C}$, with $d_p(A_n,B_n)<\epsilon_n$, such that
$$\Lambda_k(B_n)<\delta-J_k(A_n)+\Lambda_{k}(A_n),$$
where the jump is defined like we did in the discrete case by
$$J_k(A_n)=\int_X\frac{\lambda_k(A_n,x)-\lambda_{k+1}(A_n,x)}{2}\,d\mu.$$
Considering limits we get:
$$\underset{n\rightarrow \infty}{\lim}\Lambda_k(B_n)<\delta-\underset{n\rightarrow \infty}{\lim}J_k(A_n)+\Lambda_{k}(A).$$
Since $A$ is a continuity point of $\Lambda_k$ we obtain that $J_k(A_n)=0$ for all $k$ and all $n$ sufficiently large, i.e., $\lambda_k(A_n,x)=\lambda_{k+1}(A_n,x)$ for all $k$ and $\mu$-a.e. $x\in X$. Therefore, the linear differential system $A_n$ must have one-point spectrum for $\mu$-a.e. $x\in X$ and the same holds for $A$ because $\underset{n\rightarrow \infty}{\lim} \Lambda_k(A_n)=\Lambda_k(A)$. Once again we finalize the proof recalling that the set of continuity points of an upper semicontinuous function is a residual subset.
\end{proof}

The next result is the basic perturbation tool which allows us to interchange Oseledets directions.

\begin{lemma}\label{rot3cont}
Let be given a continuous linear differential system $A$ evolving in a closed accessible Lie subalgebra $\mathscr{T}\subseteq\mathfrak{gl}(d,\mathbb R)$ and over a flow $\varphi^t\colon X\rightarrow X$, $\epsilon>0$, $1\leq p < \infty$ and a non-periodic $x\in{X}$ (or periodic with period larger than $1$). There exists $r>0$ (depending on $\epsilon$) such that for all $\sigma\in(0,1)$, all $y\in B(x,\sigma r)$ (the ball transversal to $\varphi^t$ at $x$) and any continuous choice of a pair of vectors $u_y$ and $v_y$ in $ \mathbb{R}^{d}_{y}\setminus\{\vec0\}$:
\begin{enumerate}
 \item there exists a continuous linear differential system $B\in \mathscr{T}$, with $d_p(A,B)<\epsilon$ such that $\Phi^{1}_{B}(y)u_y=\Phi^{1}_{A}(y) \mathbb{R}v_y$, where $\mathbb{R}v_y$ stands for the direction of the vector $v_y$; Moreover,
 \item  there exists a traceles system $H$, supported in the flowbox $\mathcal{F}:=\{\varphi^t(y)\colon t\in[0,1], y\in B(x, r)\}$, such that $\|H\|_p<\epsilon$, $B(y)=A(y)+H(y)$ for all $y\in B(x,\sigma r)$, and $B(z)=A(z)$ if $z\notin\mathcal{F}$.
 \end{enumerate}\end{lemma}
\begin{proof}
We begin by taking $K:=\max_{z\in X}\|(\Phi_{A}^{t}(z))^{\pm 1}\|$ for $t\in[0,1]$. For a given small $r>0$ we take the closed ball centered in $x$ and with radius $r$ transversal to the flow direction and denoted by $B(x,r)$. We fix $\sigma\in(0,1)$. Let $\eta\colon \mathbb{R} \rightarrow [0,1]$ be a $C^{\infty}$ function such that $\eta(t)=0$ for $t\leq 0$ and $\eta(t)=1$ for $t\geq 1$. Let also $\rho\colon \mathbb{R} \rightarrow [0,1]$ be a $C^{\infty}$ function such that $\rho(t)=0$ for $t\leq \sigma$ and $\rho(t)=1$ for $t\geq 1$. In what follow, for $y\in B(x,r)$ we are going to define the 1-parameter family of linear maps $\Psi^{t}(y)\colon \mathbb{R}^{d}_{y} \rightarrow \mathbb{R}^{d}_{y}$ for $t\in[0,1]$.

For $t\in[0,1]$ we let $u_y^t=(1-\eta(t))u_y+\eta(t)v_y$ and, by the transitive property, we choose a smooth family $\{\mathcal{R}^t_y\}_{t\in[0,1]}$ such that $\mathcal{R}^t_y\in \mathscr{T}$ and $\mathcal{R}^t_y \,u_y=u_y^t$. Let $L>0$ be sufficiently large in order to get $\|\dot{\mathcal{R}_y^{t}}(\mathcal{R}^{t}_y)^{-1}\|<L$ for all $t\in[0,1]$ and $y\in B(x,r)$. Finally, we normalize the volume by taking $\mathfrak{R}_y^t=\zeta(t,y)\mathcal{R}^t_y$ such that $\det(\mathfrak{R}_y^t)=1$ for all $t\in[0,1]$ and $y\in B(x,r)$. Now, we take $\kappa>0$ such that $\zeta(t,y)>\kappa$ and $\dot{\zeta}(t,y)=\frac{d\zeta (t,y)}{dt}<\kappa^{-1}$ for all $t\in[0,1]$ and $y\in B(x,r)$.

Then, we consider the 1-parameter family of linear maps $\Psi^{t}(y)\colon \mathbb{R}^{d}_{y} \rightarrow \mathbb{R}^{d}_{\varphi^{t}(y)}$ where $\Psi^{t}(y)=\Phi_{A}^{t}(y) \mathfrak{R}^{t}_y$. In order to simplify the heavy notation we consider $\mathfrak{R}^t=\mathfrak{R}^t_y$, $\mathcal{R}^t=\mathcal{R}^t_y$, $\Phi^t_A=\Phi^t_A(y)$, $\zeta=\zeta(t,y)$ and $\dot{\zeta}=\frac{d\zeta(t,y)}{dt}$. We take time derivatives and we obtain:
\begin{eqnarray*}
\dot{\Psi}^{t}(y)&=& \dot{\Phi}_{A}^{t}\mathfrak{R}^{t}+\Phi_{A}^{t}\dot{\mathfrak{R}}^t=A(\varphi^{t}(y))\Phi_{A}^{t}\mathfrak{R}^{t}+\Phi_{A}^{t}\dot{\zeta}{\mathcal{R}}^t+\Phi_{A}^{t}{\zeta}\dot{\mathcal{R}}^t=\\
&=& A(\varphi^{t}(y))\Psi^{t}(y)+[\Phi_{A}^{t}\dot{\zeta}\zeta^{-1}(\Phi^t_A)^{-1}+\Phi_{A}^{t}\zeta\dot{\mathcal{R}}^{t}(\Psi^{t}(y))^{-1}]\Psi^{t}(y)\\
&=& \left[A(\varphi^{t}(y))+H(\varphi^{t}(y))\right]\cdot \Psi^{t}(y).
\end{eqnarray*}
Hence, we define, for all $y\in B(x,r)$ and $t\in[0,1]$, the perturbation $H$ in the \emph{flowbox coordinates} $(t,y)$ by
\begin{eqnarray*}
H(\varphi^{t}(y))&=&\Phi_{A}^{t}\dot{\zeta}\zeta^{-1}(\Phi^t_A)^{-1}+\Phi_{A}^{t}\zeta\dot{\mathcal{R}}^{t}(\Psi^{t}(y))^{-1}\\
&=&\frac{\dot{\zeta}}{\zeta}\textrm{Id}+\Phi_{A}^{t}\zeta\dot{\mathcal{R}}^{t}(\Phi_{A}^{t} \mathfrak{R}^{t})^{-1}\\
&=&\frac{\dot{\zeta}}{\zeta}\textrm{Id}+\Phi_{A}^{t}\dot{\mathcal{R}}^{t}(\mathcal{R}^{t})^{-1}(\Phi_{A}^{t} )^{-1}.
\end{eqnarray*}
By Jacobi's formula on the derivative of the determinant we have
\begin{eqnarray*}
\frac{d(\det (\zeta\mathcal{R}^t))}{dt}&=&\text{Tr}\left(\text{adj} (\zeta\mathcal{R}^t)\frac{d(\zeta\mathcal{R}^t)}{dt}\right)=\text{Tr}\left(\det (\zeta\mathcal{R}^t)(\zeta\mathcal{R}^t)^{-1}\frac{d(\zeta\mathcal{R}^t)}{dt}\right)\\
&=&\text{Tr}(\zeta^{-1}(\mathcal{R}^t)^{-1}(\dot\zeta\mathcal{R}^t+\zeta\dot{\mathcal{R}}^t))=\text{Tr}\left(\frac{\dot{\zeta}}{\zeta}\textrm{Id}+(\mathcal{R}^t)^{-1}\dot{\mathcal{R}}^t\right)\\
&=&\text{Tr}\left(\frac{\dot{\zeta}}{\zeta}\textrm{Id}\right)+\text{Tr}[(\mathcal{R}^t)^{-1}\dot{\mathcal{R}}^t].
\end{eqnarray*}
But we also have, for all $t\in[0,1]$ and $y\in B(x,r)$, $\det (\zeta\mathcal{R}^t)=1$ and so
$$\text{Tr}\left(\frac{\dot{\zeta}}{\zeta}\textrm{Id}+(\mathcal{R}^t)^{-1}\dot{\mathcal{R}}^t\right)=\text{Tr}\left(\frac{\dot{\zeta}}{\zeta}\textrm{Id}+\dot{\mathcal{R}}^t  (\mathcal{R}^t)^{-1}\right)=0.$$
Since the trace is invariant by any change of coordinates we obtain $\text{Tr}(H(\varphi^{t}(y)))=0$.

At this time, we consider the flowbox $\mathcal{F}:=\{\varphi^t(y)\colon t\in[0,1], y\in B(x,r)\}$ and we are able to define the linear  continuous differential system
\begin{equation}\label{B}B(z)= \left\{\begin{array}{ccc} A(z), \,\,\,\,\,\,\,\,\,\,\,\,\,\,\,\,\,\,\,\,\,\,\,\,\,\,\,\,\,\,\,\,\,\,\,\,\,\,\,\,\,\,\,\,\,\,\,\,\,\,\,\,\,\,\,\,\,\,\,\,\,\,\,\,\,\,\,\,\,\,\,\,\,\,\,\,\,\,\,\,\,\,\,\,\,\,\,\,\text{if } z\notin\mathcal{F} \\
A(z)+\left(1-\rho\left(\frac{\|x-y\|}{r}\right)\right)H(z),
\,\,\,\,\,\,\,\,\,\,\,\,\,\,\,\,\,\,\,\,\,\,\,\,\,\,\,\,\,\,\,\,\,\,\,\,\,\,\,\,\,\,\,\,\,\text{if } z=\varphi^t(y)\in\mathcal{F} \end{array}\right..\end{equation}
In order to estimate $d_p(A,B)$ it suffices to compute the $L^p$ infinitesimal generator norm of $H$. For that we consider Rokhlin's theorem (see ~\cite{R}) on disintegration of the measure $\mu$ into a measure $\hat{\mu}$ in the transversal section and the length in the flow direction, say $\mu=\hat{\mu}\times dt$. Go back into the beginning of the proof and pick $r>0$ such that
$$\hat{\mu}(B(x,r))<\left(\frac{\epsilon}{\kappa^{-2}+K^2L}\right)^p.$$
We have then \begin{eqnarray*}
\|H\|_p&=&\left(\int_\mathcal{F} \|H(z)\|^p d\mu(z)\right)^{1/p}\\
&=&\left(\int_0^1\int_{B(x,r)} \|H(\varphi^t(y))\|^p d\hat{\mu}(y)dt\right)^{1/p}\\
&=&\left(\int_0^1\int_{B(x,r)} \left\|\frac{\dot{\zeta}(t,y)}{\zeta(t,y)}\textrm{Id}+\Phi_{A}^{t}(y)\dot{\mathcal{R}}^{t}(\mathcal{R}^{t})^{-1}(\Phi_{A}^{t}(y) )^{-1}\right\|^p d\hat{\mu}(y)dt\right)^{1/p}\\
&\overset{{\tiny Minkowski }}{\leq}&\left(\int_0^1\int_{B(x,r)} \left\|\frac{\dot{\zeta}(t,y)}{\zeta(t,y)}\textrm{Id}\right\|^p\right)^{1/p}\\
&~&+\left(\int_0^1\int_{B(x,r)}\left\|\Phi_{A}^{t}(y)\dot{\mathcal{R}}^{t}(\mathcal{R}^{t})^{-1}(\Phi_{A}^{t}(y) )^{-1}\right\|^p d\hat{\mu}(y)dt\right)^{1/p}\\
&\leq&(\kappa^{-2}+K^2L)\hat{\mu}(B(x,r))^{1/p}<\epsilon.
\end{eqnarray*}

Note that the perturbed system $B$ generates the linear flow $\Phi_{A+H}^{t}(y)$ which is the same as $\Psi^{t}$ by unicity of solutions with the same initial conditions, hence given $u_y\in \mathbb{R}^d_{y}$ we have
$$\Phi_{B}^{t}(y) u_y =\Psi^{t}(y) u_y=\Phi_{A}^{t}(y)\mathfrak{R}^t_y\,u_y=\zeta(t,y)\Phi_{A}^{t}(y)\mathcal{R}^t_y\,u_y=\zeta(t,y)\Phi_{A}^{t}(y) u^t_y.$$
To finish the proof, we take $t=1$ and obtain
$$\Phi_{B}^{1}(y) u_y =\zeta(1,y)\Phi_{A}^{1}(y) u_y^1=\Phi_{A}^{1}(y) [\zeta(1,y)\,v_y].$$
\end{proof}

\medskip

\begin{remark}\label{rot33cont}
Using Lemma~\ref{rot3cont} we can also ``view'' the exchange of directions in $\mathbb{R}^{d}_{\varphi^{1}(y)}$ instead of in $\mathbb{R}^{d}_{y}$. Hence, for any two vectors $u_y^1$ and $v_y^1$ in $ \mathbb{R}^{d}_{\varphi^1(y)}\setminus\{\vec0\}$ and defining $u^0_y:=\Phi_A^{-1}(\varphi^1(y))u_y^1$, $v_y^0:=\Phi_A^{-1}(\varphi^1(y))v_y^1$, we get $\Phi^{1}_{A+H}(y)u_y^0=\Phi^{1}_{A}(y) \mathbb{R}v_y^0$, where $\mathbb{R}v_y^0$ stands for the direction of the vector $v_y^0$. Moreover, if the choice of a pair of vectors $u_y$ and $v_y$ in $ \mathbb{R}^{d}_{y}\setminus\{\vec0\}$ is only measurable, then the linear differential system $B\in \mathscr{T}$ satisfying (1) and (2) of Lemma~\ref{rot3cont} do not need to be continuous.

\end{remark}

\subsection{Simple spectrum is dense}\label{simple cont}

In this section we will obtain the continuous-time counterpart of section \S\ref{proof simple}. For that we must develop a perturbation 	
implement in the language of differential equations which plays the role of the cocycle $C_2$ in the proof of Lemma~\ref{split spec}. This is precisely what next result assures.

\begin{lemma}\label{rot4}
Let be given a continuous linear differential system $A$ evolving in a closed Lie accessible subalgebra $\mathscr{S}\subseteq\mathfrak{gl}(d,\mathbb R)$ which displays the saddle-conservative property and over a flow $\varphi^t\colon X\rightarrow X$, $\epsilon>0$, $1\leq p < \infty$ and a non-periodic $x\in{X}$ (or periodic with period larger than $1$). There exists $r>0$, such that for all $\sigma\in(0,1)$, all $y\in B(x,\sigma r)$, any $\delta>0$ and any continuous choice of directions $e_y\in \mathbb{R}^d_y$
\begin{enumerate}
 \item there exists a continuous linear differential system $B\in \mathscr{S}$, with $d_p(A,B)<\epsilon$ such that $\Phi^{1}_{B}(y)\, e_y=(1+\delta)\Phi^{1}_{A}(y)\, e_y$; Moreover,
 \item  there exists a traceless system $H$, supported in the flowbox $\mathcal{F}:=\{\varphi^t(y)\colon t\in[0,1], y\in B(x, r)\}$, such that $\|H\|_p<\epsilon$, $B(y)=A(y)+H(y)$ for all $y\in B(x,\sigma r)$, and $B(z)=A(z)$ if $z\notin\mathcal{F}$.
 \end{enumerate}
\end{lemma}

\begin{proof}
We will perform the continuous perturbations along a time-one segments of time-one orbits of $y\in B(x,r)$ for some sufficiently thin flowbox. The construction is similar to the one we did in the proof of Lemma~\ref{rot3cont}. Take $K:=\max_{z\in X}\|(\Phi_{A}^t(z))^{\pm 1}\|$ for $t\in[0,1]$.

Let $\mathcal{S}\subseteq \GL(d,\mathbb{R})$ be the saddle-conservative Lie subgroup associated to $\mathscr{S}$, $y\in B(x,r)$ and $e_y\in \mathbb{R}^d_y$ varying continuously with $y$. Fix $\delta>0$ and let $\eta\colon \mathbb{R} \rightarrow [0,1]$ be any $C^{\infty}$ function such that $\eta(t)=0$ for $t\leq 0$ and $\eta(t)=\delta$ for $t\geq 1$. Take a smooth family $\{\mathcal{E}^{t}_y\}_{t>0}\subset \mathcal{S}$ such that:
\begin{enumerate}
\item [(i)] $\mathcal{E}^{t}_y \in \SL(d,\mathbb R)$ and
\item [(ii)] $\mathcal{E}_y^{t}\, e_y=(1+\eta(t))e_y$.
\end{enumerate}

Consider the $1$-parameter family of linear maps $\Psi^{t}(y)\colon \mathbb{R}^{d}_{y} \rightarrow \mathbb{R}^{d}_{\varphi^{t}(y)}$ where $\Psi^{t}(y)=\Phi_{A}^{t}(y) \mathcal{E}^{t}_y$. We take time derivatives and we obtain:
\begin{eqnarray*}
\dot{\Psi}^{t}(y)&=& \dot{\Phi}_{A}^{t}(y)\mathcal{E}_y^{t}+\Phi_{A}^{t}(y)\dot{\mathcal{E}}^t_y=A(\varphi^{t}(y))\Phi_{A}^{t}(y)\mathcal{E}^{t}_y+\Phi_{A}^{t}(y)\dot{\mathcal{E}}^t_y\\
&=& A(\varphi^{t}(y))\Phi_{A}^{t}(y)\mathcal{E}^{t}_y+\Phi_{A}^{t}(y)\dot{\mathcal{E}}^t_y (\mathcal{E}^{t}_y)^{-1} (\Phi_{A}^{t}(y))^{-1}\Phi_{A}^{t}(y)\mathcal{E}^{t}_y\\
&=& [A(\varphi^{t}(y))+\Phi_{A}^{t}(y)\dot{\mathcal{E}}_y^t (\mathcal{E}^{t})_y^{-1} (\Phi_{A}^{t}(y))^{-1}]\Psi^{t}(y).
\end{eqnarray*}
The perturbation is then defined by:
$$H(\varphi^t(y))=\Phi_{A}^{t}(y)\dot{\mathcal{E}}_y^t (\mathcal{E}^{t}_y)^{-1} (\Phi_{A}^{t}(y))^{-1}.$$
We can define now the continuous linear differential system $B$ as in \eqref{B}.
Now it is time to choose the thickness $r>0$.  Let $L>0$ be such that $\|\dot{\mathcal{E}}_y^{t}(\mathcal{E}_y^{t})^{-1}\|^p\leq L$, for all $y\in B(x,r)$ and $t\in[0,1]$. Finally, take $r>0$ such that:
$$\hat{\mu}(B(x,r))<\left(\frac{\epsilon}{L K^2}\right)^p.$$
To estimate $d_p(A,B)\leq \|H\|_p$, we have
\begin{eqnarray*}
\|H\|_p&=&\left(\int_\mathcal{F} \|H(z)\|^p d\mu(z)\right)^{1/p}\\
&=&\left(\int_0^1\int_{B(x,r)} \|H(\varphi^t(y))\|^p d\hat{\mu}(y)dt\right)^{1/p}\\
&=&\left(\int_0^1\int_{B(x,r)} \left\|\Phi_{A}^{t}(y)\dot{\mathcal{E}}_y^t (\mathcal{E}_y^{t})^{-1} (\Phi_{A}^{t}(y))^{-1}\right\|^p d\hat{\mu}(y)dt\right)^{1/p}\\
&\leq&L K^2\hat{\mu}(B(x,r))^{1/p}<\epsilon.
\end{eqnarray*}
Finally, we observe that
$$\Phi^{1}_{B}(y)\, e_y=\Psi^{1}(y)\,e_y=\Phi_{A}^{1}(y) \mathcal{E}^{1}_y \,e_y=\Phi_{A}^{1}(y) (1+\eta(1))\,e_y=(1+\delta)\Phi^{1}_{A}(y) e_y.$$
\end{proof}

The proof of Theorem~\ref{simple2}, which asserts the density of cocycles with simple spectrum in continuous-time cocycles, follows by similar arguments as the proof of Theorem~\ref{simple}. Since Lemma~\ref{lit change}~\cite[Lemma 4.3]{AC0} holds trivially for continuous-time cocycles, we only need the flow version of Lemma~\ref{split spec}, which we write down for completeness:

\begin{lemma}\label{split spec flow}
Assume that $A\in \mathcal{S}_{\!I\!C}$ over a flow $\varphi^t\colon X\rightarrow X$ has one-point
spectrum and $d\geq2$. Then, for any small $\epsilon >0$ and $1\leq p < \infty$, there exists $B\in\mathcal{S}_{\!I\!C}$, with
$\|A-B\|_p<\epsilon$, such that $B$ has at least two different
Lyapunov exponents.
\end{lemma}

\begin{proof}
We will consider $A$ to be continuous because we can always approximate, in the $L^p$-sense, the linear differential system $A$ by another one which is continuous. Consider a transversal section to the flow $\Sigma\subset X$, such that the time-one flowbox $V:=\varphi^{[0,1]}(\Sigma)$ is such that $\mu(V)>0$,
 $V\cap \varphi^1(V)=\varphi^1(\Sigma)$. Let $L_A:=\max_{x\in X}\|A(x)\|$  and
$k(x):=\min\{ t>0: \varphi^{-t}(x)\in \varphi^1(\Sigma)\}$. For the sake of simplicity of presentation we assume that $\Sigma$ is a transversal closed ball $B(p,r)$. Fix a unitary vector $e\in
\mathbb RP^{d-1}$ and define the following \emph{vector field} which is a
normalized image under the cocycle associated to $A$ of the direction associated to the vector $e$, in the
fiber corresponding to each $x\in X$:
 $$v(x):=\left\{\begin{array}{lll}
e,&& \textrm{if}\, x\in \varphi^1(\Sigma)\\
\frac{\Phi_A^{k(x)}(\varphi^{-k(x)}(x))e}{\|\Phi_A^{k(x)}(\varphi^{-k(x)}(x))e\|},&& \textrm{otherwise}
\end{array}\right., $$ and set $E(x)=\textrm{span}\{v(x)\}$.
For $x\in \Sigma$ define $q(x)\in\mathbb RP^{d-1}=\mathbb{R}^d_x$ given
by $$q(x)=\frac{\Phi^1_A(x)v(x)}{\|\Phi^1_A(x)v(x)\|}.$$

Let $H_{q(\cdot)}\colon X\rightarrow{{\mathfrak {sl}}(d,\mathbb{R})}$ be a linear differential system, supported in $V$ and constructed following the steps of
Lemma \ref{rot3cont} and Remark~\ref{rot33cont}, such that, if  $x\in \Sigma$ and $e\notin \langle q(x)\rangle$ we have:
$$\Phi^1_{A+H_{q(x)}}(x)v(x)=\Phi^1_A(x)\mathbb{R}w(x),$$
where $w(x):=[\Phi^1_A(\varphi^1(x))]^{-1} e$.
Clearly, $d_p(A,A+H_{q})$ can be smaller than any small
$\epsilon>0$ just considering $V$ with small enough $\mu$-measure. If
$A+H_{q}$ has two or more distinct Lyapunov exponents we take $B=A+H_{q}$
and we are done.

Let us consider now that $A+H_{q}$ has only one Lyapunov exponent
$\lambda_{A+H_{q}}$. Then, it must be equal to the unique Lyapunov
exponent $\lambda_A$ for $\Phi^1_A$ (and both have multiplicity $d$).
Indeed, by Os\-tro\-grad\-sky-Jacobi-Liouville formula in (\ref{OJL}) we get
$$\det \Phi^t_A(x)=\det \Phi^t_{A+H_{q(x)}}(x)$$ for all $x\in X$, and by the
multiplicative ergodic theorem we have
$$d.\lambda_{A+H_{q}}=\int \log |\det {\Phi^1_{A+H_{q(x)}}}(x)|\,d\mu=\int \log |\det
\Phi^1_A(x)|\,d\mu=d.\lambda_A.$$

Fix $\delta\in(0,1)$. Since our algebra has the saddle-conservative property,  we let $J\colon X\rightarrow{{\mathfrak {sl}}(d,\mathbb{R})}$ be a linear differential system, supported in $\varphi^1{(V)}$ and constructed following the steps of
Lemma \ref{rot4}, such that, for $x\in \varphi^1{(\Sigma)}$, we have
$$\Phi^1_{A+J}(x)e=(1+\delta)\Phi^1_A(x)e.$$
Finally, define the continuous linear differential system, supported in $\varphi^{[0,2]}(V)$, by
$$D(x)=A(x)+H_{q(x)}(x)+J(x).$$ Since, for all $x\in X$
\begin{equation*}
\Phi^1_{D}(x)E(x)=\Phi^1_{A+H_{q(x)}}(x)E(x)=E(\varphi^1(x))
\end{equation*}
by Birkhoff's ergodic theorem we have:
\begin{align}\lambda(D,x,v(x))&:=\lim_{t\to\infty}\frac1t\log\|\Phi_{D}^t(x)v(x)\|\nonumber\\
&=\lim_{n\to\infty}\frac1n\log\|\Phi_{D}^n(x)v(x)\|\nonumber\\&=\lim_{n\to\infty}\frac1n\log\|(1+\delta)^{\sum_{j=0}^{n-1}\mathbbm
l_{V}(\varphi^j(x))}\Phi_{A+H_{q(x)}}^n(x)v(x)\|\nonumber\\
& = \lambda(A+H_{q(x)},x,v(x)) + \log(1+\delta)\mu(V)\label{lyap exponent
for D2flow}.
\end{align}
Let $\lambda_{D,1}\geq
\lambda_{D,2}\geq\ldots\geq\lambda_{D,r_\delta}$ be the distinct
Lyapunov exponents for $D$, with the corresponding multiplicities
$m_1, \ldots, m_{r_\delta}$. Since for all $x\in X$,
$$ \det \Phi_{D}^1(x)=\det \Phi^1_{A+H_{q(x)}}(x)=\det
\Phi^1_A(x)\label{dets},
$$
by the multiplicative ergodic theorem we also have
$$
\sum_{i=1}^{r_\delta}\lambda_{D,i}\cdot m_{i}=d.\lambda_A.
$$ By \eqref{lyap exponent
for D2flow}, for any $\delta>0$ the linear differential system $D$ has a Lyapunov exponent
equal to $\lambda_A+\log(1+\delta)\mu(V)$, so we must have
$r_\delta\geq2$. Moreover, for all $\delta>0$, we have
\begin{itemize}
\item $J$ is supported in $\varphi^1(V)$ and is bounded;
\item $H_q$ is supported in $V$ and is bounded and so,
\item $D(x)=A(x)$ in $x\notin V\cup \varphi^1(V)$ and $D$ is bounded.
\end{itemize}
which implies that $d_p(A,D)$ can be made as small as we want by decreasing $r>0$. We just have now to consider $B=D$.
\end{proof}

\section{Applications to discrete systems}\label{app}

\subsection{Dynamical cocycles}\label{dynamical}

We would like to present now an application to the so-called \emph{dy\-na\-mi\-cal cocycle}. In this case we consider that the base dynamics and the fiber dynamics are related. In fact, the fibered action is given by the tangent map on the tangent bundle of the action defined in the base. Of course that these systems are much more delicate than the ones studied along this paper since the perturbations in the fiber have to be obtained by the effect of a perturbation in the base. Let us present briefly the setting we are interested in. From know on we let $M$ be a closed Riemannian surface and $\mu$ the Lebesgue measure arising from the area-form in $M$. Let $\text{Hom}_{\mu}(M)$ stands for the set of homeomorphisms in $M$ which keep the Lebesgue measure invariant and $\text{Diff}^1_\mu(M)$ the set of diffeomorphisms of class $C^1$ supported on $M$. Finally, we let $\text{Hom}_{\mu}^p(M)$ denote the set of elements $f\in\text{Hom}_{\mu}(M)$ such that for $\mu$-a.e. $x\in M$ the map $f$ has well defined derivative $Df(x)$ and it is $L^p$-integrable, i.e,
$$\left(\int_M \|Df(x)\|^p\,d\mu\right)^{1/p}<\infty.$$
Moreover, we topologize $\text{Hom}_{\mu}^p(M)$ with the topology (denominated by $L^p$-topology) defined by the maximum of the $C^0$-topology (cf. ~\cite{BS}) and the one analog to the one constructed in \S\ref{topologies}. Then, we take the $L^p$-completement of $\text{Hom}_{\mu}^p(M)$ which we still denote by $\text{Hom}_{\mu}^p(M)$. By Baire's category theorem $\text{Hom}_{\mu}^p(M)$ is a Baire space.

Each map $f\in \text{Diff}^1_\mu(M)$ generates a linear (dynamical) cocycle
$F_f\colon TM\to TM$ generated by:
$${F_f}(x,v)=(f(x), Df(x)v),$$
and the same holds for $f\in \text{Hom}_{\mu}^p(M)$ at least for a full measure subset $\hat{M}\subseteq M$. Since these maps preserve the Lebesgue measure $Df(x)\in\SL(2,\mathbb{R})$.

From now on we endow $\text{Hom}_{\mu}(M)$ with the $C^0$-topology, $\text{Hom}_{\mu}^p(M)$ with the $L^p$-topology and $\text{Diff}^1_\mu(M)$ with the $C^1$-Whitney topology.

In ~\cite{BS} it was proved that $C^0$-densely elements in $\text{Hom}_{\mu}(M)$ have one-point spectrum. On the other hand, in ~\cite{B}, it was proved that $C^1$-generic elements in $\text{Diff}^1_\mu(M)$ are Anosov or else have one-point spectrum. Here, we describe what behavior occurs in the middle:

\begin{maintheorem}\label{dc}
There exists an $L^p$-residual subset $\mathcal{R}$ of $\text{Hom}_{\mu}^p(M)$,  $1\leq p < \infty$ such that, for any $f\in \mathcal{R}$ we have that $\mu$-a.e. $x\in M$ has all Lyapunov exponents equal to zero.
\end{maintheorem}
Let us now see the highlights of the proof of previous theorem.

\medskip

\emph{(i) On the entropy function:}

Given a set of measurable and Lebesgue invariant maps $\mathscr{T}$ endowed with a certain topology $\tau$ we consider the function that associated to each $f\in \mathscr{T}$ the integral over $M$ of its upper Lyapunov exponent with respect to the Lebesgue measure:
$$
\begin{array}{cccc}
\Lambda\colon & (\mathscr{T},\tau) & \longrightarrow & [0,\infty[ \\
& f & \longmapsto & \int_M \lambda_1(f,x)\,d\mu,
\end{array}
$$
It was proved in ~\cite[\S 4]{BS} that when $\mathscr{T}=\text{Hom}_{\mu}(M)$ and $\tau$ the $C^0$-topology, then $\Lambda$ cannot be upper semicontinuous. Moreover, in \cite[Proposition 2.1]{B} is was proved that when $\mathscr{T}=\text{Diff}_{\mu}^1(M)$ and $\tau$ the $C^1$-topology, then $\Lambda$ is upper semicontinuous. When $\mathscr{T}=\text{Hom}_{\mu}^p(M)$ and $\tau$ the $L^p$-topology, then $\Lambda$ is upper semicontinuous by using the arguments described in ~\cite{AB} which, we recall, do not require $f$ to be ergodic.

\medskip

\emph{(ii) On the perturbations:}

In ~\cite[\S3.1]{B} it was developed the concept of \emph{realizable sequences} (in the $C^1$-sense) and in ~\cite[\S 2.4]{BS} the  the concept of \emph{topological realizable sequences} (in the $C^0$-sense). Here, we need an $L^p$-version of it. Then, since we can rotate any angle we like, on the action of $Df$, by making an arbitrarily small $L^p$-perturbation the uniform hyperbolicity cannot be an obstacle in order to decay the Lyapunov exponent as it is in Bochi's setting. Therefore, we can proceed like in \cite{BS} and obtain a map with arbitrarily small Lyapunov exponent near any map (even an Anosov one). Recall the points (1), (2) and (3) in \S \ref{BVrevisited}. Once again we emphasize that the use of Bochi's strategy is crucial because Arnold and Cong's arguments assume the ergodicity of the base map and in our dynamical cocycle context the base dynamics change and may eventually be non ergodic\footnote{ We observe that, despite the fact that Oxtoby and Ulam theorem (\cite{OU}) assures that $C^0$-generic volume-preserving maps are ergodic, the set of $C^0$-stably ergodic (and also $L^p$-stably ergodic) ones is empty.}.

\medskip

\emph{(iii) End of the proof:}

We pick a point of continuity $f$ of the function $\Lambda\colon (\text{Hom}_{\mu}^p(M),L^p)\rightarrow [0,\infty[$. We claim that $\Lambda(f)=0$ otherwise, if $\Lambda(f)=\alpha>0$, then, by (ii) we consider $g\in \text{Hom}_{\mu}^p(M)$ arbitrarily $L^p$-close to $f$ and such that $\Lambda(g)=0$ which contradicts the fact that $f$ is a continuity point of $\Lambda$. Finally, we use (i), and the fact that the points of continuity of an upper semicontinuous function if a residual subset.

\medskip

\emph{(iv) A final remark:}

Other strategy which simplify considerably the previous argument needs to assume that $\text{Diff}_{\mu}^1(M)$ is $L^p$-dense in $\text{Hom}_{\mu}^p(M)$. First, we approximate by a $C^1$-diffeomorphism $f$, and then reasoning in the following way using Bochi's theorem: if $f$ has all its Lyapunov exponent equal to zero we are over arguing like we did before using (i). Otherwise, $f$ is Anosov (or in the $C^1$-boundary of it), and a small $L^p$-perturbation send us to the interior of the non-Anosov ones (Anosov is no longer open w.r.t. the $L^p$-topology).

\subsection{Infinite dimensional discrete cocycles}\label{infinite}

We denote by $\mathscr{H}$ an infinite dimensional separable Hilbert space
and by $\mathcal{C}(\mathscr{H})$ the set of linear compact operators
acting in $\mathscr{H}$ endowed with the uniform operators norm. We fix a map
 $T:X\rightarrow{X}$ as before and $\mu$ an $f$-invariant Borel regular measure that is positive on
non-empty open subsets. Given a family $(A_{x})_{x \in X}$ of
operators in $\mathcal{C}(\mathscr{H})$ and a continuous vector
bundle $\pi: X \times \mathscr{H} \rightarrow {X}$, we define the
cocycle by
$$
\begin{array}{cccc}
F_A: & X\times{\mathscr{H}} & \longrightarrow & X\times{\mathscr{H}} \\
& (x,v) & \longmapsto & (T(x),A(x) v).
\end{array}
$$
It holds $\pi \circ {F}=f\circ{\pi}$ and,
for all $x\in X$, $F_{A}(x,\cdot):\mathscr{H}_x\rightarrow{\mathscr{H}_{f(x)}}$ is a
linear operator. We let $C^0_I(X,\mathcal{C}(\mathscr{H}))$ stand for the continuous integrable cocycles evolving in $\mathcal{C}(\mathscr{H})$ and endowed with the $C^0$-topology. Let also $L^p_I(X,\mathcal{C}(\mathscr{H}))$ stand for the continuous integrable cocycles evolving in $\mathcal{C}(\mathscr{H})$ and endowed with the $L^p$-topology.

These infinite dimensional cocycles display some properties similar to the ones in finite dimension. For instance, the existence of an asymptotic spectral decomposition with asymptotic uniform rates like the ones given in Oseledets theorem also holds by an outstanding result by Ruelle (see ~\cite{Ru}). Moreover, in ~\cite{BeC} was obtained the Ma\~n\'e-Bochi-Viana dichotomy for $C^0_I(X,\mathcal{C}(\mathscr{H}))$ equipped with the $C^0$-topology. Here, we intend to get the $L^p$-version of ~\cite{BeC} for $L^p_I(X,\mathcal{C}(\mathscr{H}))$ cocycles with the $L^p$ topology. We point out that such infinite dimensional systems have been the focus of attention (cf. \cite{BeC,BeC2,LY,LY2}) not only because of its intrinsic interest but also due to its potential applications to partial differential equations (see ~\cite[\S1.3 and \S2]{LY}).

As is expected we do drop the dichotomy in ~\cite[Theorem 1.1]{BeC} and reach the one-point spectrum statement.

\begin{maintheorem}\label{BeCLp}
There exists a $L^{p}$-residual subset $\mathcal{R}$ of the set of
integrable compact cocycles
${L_{I}^{p}(X,\mathcal{C}(\mathscr{H}))}$ such that, for
$A\in{\mathcal{R}}$ and $\mu$-almost every $x\in{X}$
$$\underset{n\rightarrow{\infty}}{\text{lim}}({A(x)^{*}}^{n}A(x)^{n})^{\frac{1}{2n}}=[0],$$
where $[0]$ stands for the null operator.
\end{maintheorem}

The strategy to obtain the proof of Theorem~\ref{BeCLp} is much like to the one described \S\ref{dynamical} which follows the three steps (i), (ii) and (iii). Once again we are free to input rotations on the fiber $\mathscr{H}_{X}$ by small $L^p$-perturbation highlighting the key point for this kind of systems. It is interesting to observe that the strategy of Arnold and Cong cannot be adapted directly to this setting. Actually, their argument is based on a \emph{finite} circular permutation on the fiber directions which have already simple spectrum (see ~\cite[Theorem 4.5]{AC0}) which we can not see how to implement to the infinite dimensional context. Once again our choice of using Bochi and Viana strategy is crucial to obtain our results.
\bigskip

\textbf{Acknowledgements:} The authors were partially supported by National Funds through FCT - ``Funda\c{c}\~{a}o para a Ci\^{e}ncia e a Tecnologia'', project PEst-OE/MAT/UI0212/2011. We would like to thank Borys Alvarez-Samaniego for some suggestions given.

\vspace{0.5cm}

\begin{tiny}

\noindent
\begin{minipage}[t]{.4\linewidth}
M\'ario Bessa

Universidade da Beira Interior, 225

Rua Marqu\^es d'\'Avila e Bolama,

6201-001 Covilh\~a

Portugal.

bessa@ubi.pt

\end{minipage}
\hspace{.2\linewidth}
\begin{minipage}[t]{.4\linewidth}
Helder Vilarinho

Universidade da Beira Interior, 225

Rua Marqu\^es d'\'Avila e Bolama,

6201-001 Covilh\~a

Portugal.

helder@ubi.pt

\end{minipage}

\end{tiny}

\end{document}